\documentclass[11pt]{amsart}
\usepackage{layout}
\usepackage{geometry}
\numberwithin{equation}{section}

\usepackage{bbm}
\usepackage{comment}
\usepackage{verbatim}
\usepackage{mathrsfs}
\usepackage{mathtools}
\usepackage{latexsym}
\usepackage{epsfig}
\usepackage{amsmath}
\usepackage[usenames,dvipsnames]{xcolor}
\usepackage{bbm}
\usepackage{graphics}
\usepackage[utf8]{inputenc}
\usepackage{amssymb}
\usepackage{amsthm}
\usepackage[alphabetic]{amsrefs}
\usepackage{amsopn}
\usepackage{amscd}
\usepackage[all,knot]{xy}
\xyoption{all}
\usepackage{rotating}
\usepackage{tikz}
\usetikzlibrary{calc}
\usepgflibrary{shapes.geometric}
\usepgflibrary{shapes.misc}
\usetikzlibrary{positioning}
\usetikzlibrary{decorations}
\usetikzlibrary{decorations.pathreplacing}
\usepackage{hyperref}
\usepackage{tikz}
\usepackage{tikz-cd}
\usepackage{enumitem}
\mathtoolsset{showonlyrefs}

\theoremstyle{plain}
\newtheorem{thm}{Theorem}[section]
\newtheorem{lem}[thm]{Lemma}
\newtheorem{prop}[thm]{Proposition}
\newtheorem{cor}[thm]{Corollary}
\newtheorem{ques}[thm]{Question}

\newtheorem*{thm*}{Theorem}
\newtheorem*{lem*}{Lemma}
\newtheorem*{prop*}{Proposition}
\newtheorem*{cor*}{Corollary}

\theoremstyle{definition}
\newtheorem{defn}[thm]{Definition}
\newtheorem*{defn*}{Definition}

\newtheorem{ex}[thm]{Example}
{}
\newtheorem{rem}[thm]{Remark}
\newtheorem*{rem*}{Remark}

{}
{}
{}
\newtheorem*{ack}{Acknowledgements}{}

\theoremstyle{remark}
{}
{}
{}

\def\to{\longrightarrow}

\def\NN{\mathbb{N}}

\DeclareUnicodeCharacter{221E}{$\infty$}

\DeclareMathOperator{\Hom}{Hom}

\DeclareMathOperator{\End}{End}
\DeclareMathOperator{\id}{id}
\DeclareMathOperator{\im}{im}

\DeclareMathOperator{\cha}{char}

\DeclareMathOperator{\Der}{Der}
\DeclareMathOperator{\Inn}{Inn}

\DeclareMathOperator{\gldim}{gl.dim}

\DeclareMathOperator{\smodu}{\underline{\mathsf{mod}}}

\DeclareMathOperator{\Perf}{\mathsf{Perf}}

\DeclareMathOperator{\res}{res}

\DeclareMathOperator{\rad}{rad}

\DeclareMathOperator{\HHH}{HH}

\definecolor{internationalkleinblue}{rgb}{0.0, 0.18, 0.65}

\title[On solvability of the first Hochschild cohomology \ldots]{On solvability of the first Hochschild cohomology of a finite-dimensional algebra}

\author{Florian Eisele}
\address{Florian Eisele, School of Mathematics and Statistics,
University of Glasgow,
University Place,
Glasgow G12 8QQ
}
\email{florian.eisele@glasgow.ac.uk}
\urladdr{https://feisele.github.io}

\author{Theo Raedschelders}
\thanks{The second author is supported by an EPSRC postdoctoral fellowship EP/R005214/1.}
\address{Theo Raedschelders, School of Mathematics and Statistics,
University of Glasgow,
University Place,
Glasgow G12 8QQ
}
\email{theo.raedschelders@glasgow.ac.uk}
\urladdr{http://www.maths.gla.ac.uk/~traedschelders/}


\begin{document}

\begin{abstract}

For an arbitrary finite-dimensional algebra $A$, we introduce a general approach to determining when its  first Hochschild cohomology $\HHH^1(A)$, considered as a Lie algebra, is solvable. 
If $A$ is moreover of tame or finite representation type, we are able to describe $\HHH^1(A)$ as the direct sum of a solvable Lie algebra and a sum of copies of $\mathfrak{sl}_2$. We proceed to determine the exact number of such copies, and give an explicit formula for this number in terms of certain chains of Kronecker subquivers of the quiver of $A$. As a corollary, we obtain a precise answer to a question posed by Chaparro, Schroll and Solotar. 
\end{abstract}

\maketitle
\tableofcontents

\section{Introduction}

Hochschild cohomology is one of the most intriguing invariants of an associative algebra, as can already be gleaned from the classical references~\cite{eilenberg1956homological,gerstenhaber1964deformation,loday2013cyclic}. The low-dimensional Hochschild cohomology groups have well known interpretations, and control the infinitesimal deformation theory of the algebra: a lengthy survey is given in~\cite{gerstenhaber1988algebraic}, which despite being over 30 years old, is still an excellent read.

The actual computation of Hochschild cohomology is greatly impeded by the fact that it enjoys only limited functoriality properties~\cite{keller2003derived}, and does not give rise to an ``additive'' invariant~\cite{tabuada2015noncommutative}. One case of particular interest are finite-dimensional algebras over a field, as the very well developed representation theory of such algebras may yield information on Hochschild cohomology groups, see for example~\cite{bardzell1997alternating,erdmann1999twisted,happel1989hochschild,
happel1990hochschild,holm2004hochschild}.

Hochschild cohomology enjoys a rich and complex structure: the Hochschild cochain complex carries a $B_{\infty}$-algebra structure~\cite{getzler1994operads}, which in particular induces the more well-known Gerstenhaber algebra structure on Hochschild cohomology~\cite{gerstenhaber1963cohomology}. That being said, in this article we will  focus on only a small part of this structure. Remember that the first Hochschild cohomology group $\HHH^1(A)$ of an algebra $A$ can be identified with the group of outer derivations of $A$, and the commutator of derivations endows this group with a Lie algebra structure, which is a shadow of the Gerstenhaber algebra structure mentioned above. 

Based on a mounting body of evidence,
Chaparro, Schroll and Solotar~\cite{chaparro2018lie} suggested that for a finite-dimensional algebra $A$, there should be a close connection between the Lie algebra structure on $\HHH^1(A)$ and the representation type of $A$. More precisely (see also \S\ref{css}), they showed that for a gentle algebra $A$ defined over an algebraically closed field of characteristic not equal to two, the first Hochschild cohomology of $A$ is solvable as a Lie algebra if and only $A$ is not Morita equivalent to a Kronecker quiver. Gentle algebras are of tame representation type, which led them to ask the following question:
\begin{ques}\cite[\S 1]{chaparro2018lie}
\label{qintro}
Is it true that, except in some low dimensional cases, the first Hochschild cohomology space of any finite-dimensional algebra of tame representation type is a solvable Lie algebra?
\end{ques}

\subsection{A general reduction} 
\label{reduction}
Assume now that $A$ is a finite-dimensional algebra (of arbitrary representation type) defined over an algebraically closed field $k$ of characteristic $p \geq 0$. In this paper we give a general approach to the problem of determining when the first Hochschild cohomology $\HHH^1(A)$ is solvable as a Lie algebra, and in the case of tame algebras we obtain a complete answer to Question \ref{qintro}. 

If $k$ is of characteristic zero, then every derivation of $A$ preserves the Jacobson radical $\rad(A)$, and hence every derivation of $A$ induces a derivation of $A/\rad^2(A)$. This  defines a map 
\begin{equation}
\psi_2: \HHH^1(A) \to \HHH^1(A/\rad^2(A))
\end{equation}
which is compatible with the Lie bracket. Hence $\HHH^1(A)$ is solvable if and only if $\ker(\psi_2)$ and $\im(\psi_2)$ are solvable. One can show that $\ker(\psi_2)$ is always solvable, which moves the focus to $\im(\psi_2)$.  This reduction step is a crucial first ingredient to the method we use to tackle Question~\ref{qintro}. 

Away from the characteristic zero case, the situation is a little more intricate, but a reduction is still possible.
If $k$ is of arbitrary characteristic, it is no longer true in general that every derivation preserves the radical, leading us to define $\HHH^1_{\rad}(A)$, which only takes the radical preserving derivations into account. 
We then get the following diagram of Lie algebras, which will be established over the course of \S\ref{sec:lie}--\S\ref{sec:hoco non-wild}.
\begin{equation}
\label{diagram big picture}
	\xymatrix{
		& & \HHH^1(A) \ar[rd]^-{\psi_1} \\ 
		& \ker(\psi_1)\ar@{^(->}[r] \ar@{^(->}[ru] & \HHH^1_{\rad}(A) \ar^-{\psi_2}[rd] \ar@{^(->}[u]  & \HHH^1(A/ \sum_{i\neq j} Ae_iAe_jA) \ar@{=}[r]^-\sim & \HHH^1(\bigoplus \textrm{local algebras}) \\		
		&\ker(\psi_2) \ar@{^(->}[ru] & & \HHH^1_{\rad}(A/\rad^2(A)) \ar@{->>}[rd]^{\psi_3} \\ && \ker(\psi_3) \ar@{^(->}[ru] && \HHH^1(kQ^s)
	}
\end{equation}
where the $e_i$ denote a full set of non-conjugate orthogonal primitive idempotents in $A$, $Q$ is the quiver of $A$ and $Q^s$ is the separated quiver of $A$ (see \S\ref{sec:hoco non-wild}). We show that $\ker(\psi_2)$ and $\ker(\psi_3)$ are solvable in general, thereby reducing showing solvability of $\HHH^1(A)$ to showing solvability of the images of $\psi_1$ and $\psi_3\circ \psi_2$. We emphasize that this picture is independent of the characteristic of $k$ and the representation type of $A$. 
Moreover, in Proposition \ref{prop:der-rad}, we show that as long as $p$ ``avoids the relations'' of $A$ (e.g. if the quiver of $A$ contains no loops), we have $\HHH^1(A)=\HHH^1_{\rad}(A)$,  simplifying diagram \eqref{diagram big picture} similar to the characteristic zero case. This is the case for most examples of interest, but it should nonetheless be noted that Proposition~\ref{prop:der-rad} is by no means exhaustive. 

When $A$ is of tame or finite representation type our results give an if-and-only-if criterion for the solvability of the image of $\psi_3\circ \psi_2$, independent of the characteristic of $k$ (see \S\ref{tamesec}). But even if $A$ is of wild representation type, one can still use knowledge of $\HHH^1$ of the hereditary algebra $kQ^s$ to infer solvability of $\HHH^1(A)$, for example for certain shapes of $Q$. And while we do not study the image of $\psi_1$ in detail, one can still attempt to show that the image of $\psi_1$ is solvable (or even zero) in cases one is interested in. That is, looking at $\HHH_{\rad}^1(A)$ instead of $\HHH^1(A)$ is less of a limitation  than one might think at first sight.

\subsection{Algebras of non-wild representation type} 
\label{tamesec}
Assume for the rest of the introduction that $\cha (k)\neq 2$ (see Theorem~\ref{thm:main} and Corollary~\ref{cor:char2} for results in characteristic two).
Our main result gives a complete answer to Question \ref{qintro} by characterising algebras of finite or tame (i.e. ``non-wild'') representation type with non-solvable first Hochschild cohomology in terms of their presentations. 

\begin{thm}[see Theorem \ref{thm:mainmain}]
\label{thm:main-intro}
Let $A=kQ/I$ be a finite-dimensional algebra of non-wild representation type. Then there is an isomorphism of Lie algebras 
\begin{equation}
\label{eq:levi}
\HHH^1_{\rad}(A) \cong \mathfrak{sl}_2^{\oplus m} \oplus \mathfrak{r},
\end{equation}
where $\mathfrak{r}$ is solvable and $m$ is the number of equivalence classes of maximal Kronecker chains with standard relations in $A$.
\end{thm}

A Kronecker chain with standard relations is simply a chain of Kronecker quivers embedded in $Q$, satisfying certain relations, the precise description of which we defer to the main body of the text (see Definitions \ref{def:maxkron} and \ref{def:standardrels}). What is important is that  for a fixed $A$, the number $m$ can readily be determined from a presentation, and for a non-wild algebra, the first Hochschild cohomology is solvable if and only if this number $m=0$. Given the explicit nature of Theorem \ref{thm:main-intro}, it is not difficult to construct examples of arbitrary dimension with $m \neq 0$, allowing us to answer Question \ref{qintro} in \S\ref{css}. 

Also in \S\ref{css}, we show that Theorem \ref{thm:main-intro} is strong enough to recover the results (pertaining to solvability of $\HHH^1(A)$) of \cite{chaparro2018lie,artenstein2018gerstenhaber,meinel2018gerstenhaber} and we also give some new consequences of Theorem \ref{thm:main-intro}. These are summarized in the following theorem.

\begin{thm}[see Corollary \ref{cor:fin-rep-type} and Theorem \ref{thm:tame-symmetric}]
Let $A=kQ/I$ be a finite-dimensional algebra.
\begin{enumerate}
\item If $A$ is of finite representation type, then $\HHH^1_{\rad}(A)$ is solvable.
\item If $A$ is symmetric of tame representation type, then $\HHH^1_{\rad}(A)$ is solvable if and only if $A \not\cong T(K_2)$, the trivial extension of the Kronecker algebra.
\end{enumerate}
\end{thm}

\subsection{Proof of Theorem \ref{thm:main-intro}}

While the details of the proof of Theorem~\ref{thm:main-intro} are quite intricate, there are two recurring main ideas we would like to highlight. The first idea is that an algebra $A$ being of tame or finite representation type guarantees that there are no more than two parallel arrows in the quiver of $A$. This goes not just for $A$, but also for idempotent subalgebras of the form $eAe$, and potentially quotients thereof. So one can often reduce to the case of some quite small quivers. Lemma~\ref{lemma double kronecker} is an example of a result for a fixed small quiver which we then reduce other cases to.

 The second idea comes from the fact that $\HHH^1(A)$ is related to the Lie algebra of the (outer) automorphism group of $A$. The idea is best explained using automorphism groups, even though, for technical reasons, we do not use them in the actual proof and instead work with Lie algebras throughout.
But $\HHH^1(A)$ having an $\mathfrak{sl}_2$ as a quotient should (in a rough sense) correspond to the automorphism group of $A$ having an $\mathrm{SL}_2$ or a $\mathrm{GL}_2$ as a quotient. If the automorphism group of $A$ acts irreducibly on the vector space spanned by two parallel arrows through such a quotient, then a relation in $A$ involving one of these arrows yields an entire orbit of relations involving these two parallel arrows. The typical outcome in the small cases we are reduced to is then that either (almost) every path of length two is a relation, or none of them is, as all intermediate options are not stable under the action of the automorphism group. Tameness then usually precludes the case of no relations. This idea is first used in Proposition~\ref{prop:local}, but also in Lemma~\ref{lemma double kronecker} and some other places.

\subsection{Relation to other work}
After the first version of this preprint appeared on the arXiv, two further preprints, one by Linckelmann and Rubio y Degrassi \cite{linckelmanndegrassi} and one by Rubio y Degrassi, Schroll and Solotar \cite{schrolletal2019hochschild} appeared, both of which deal with very similar questions. We do not attempt a comparison with these results.

\subsection{Conventions} 
Assume $k$ is a field of arbitrary characteristic. From \S\ref{sec:hoco non-wild} onwards we will also assume $k=\bar k$.
Algebras are always assumed to be finite-dimensional, even if not explicitly stated.
For a finite-dimensional $k$-algebra $A$, we will denote by $\rad (A)$ its Jacobson radical. Throughout, we will assume that 
$A/\rad (A)$ is separable over $k$ (this is always the case if $k$ is perfect). When we write 
``$A=kQ/I$'' for a quiver $Q$ and an ideal $I$, we will tacitly assume that $I$ is an admissible ideal (i.e. contained in the ideal generated by paths of length two), unless explicitly stated otherwise. By $Q_1$ we denote the set of arrows of $Q$, and by $Q_0$ the set of vertices.

\section{The Lie algebra structure on the first Hochschild cohomology}
\label{sec:lie}
Remember that $\Der(A)$ forms a Lie algebra under the commutator bracket $[-,-]$, and denote by
\begin{equation}
\Der_{\rad}(A)=\{\delta \in \Der(A) \mid \delta(\rad (A)) \subset \rad (A)\}
\end{equation}
the derivations preserving the radical. Note that this is a sub Lie algebra of $\Der(A)$, and moreover, the inner derivations 
\begin{equation}
\Inn(A)=\{\delta \in \Der(A) \mid \delta=[a,-] \text{ for some } a \in A\} \subset \Der_{\rad}(A)
\end{equation}
form a Lie ideal. Hence we can define 
\begin{equation}
\HHH^1_{\rad}(A)=\Der_{\rad}(A)/\Inn(A) \subset \Der(A)/\Inn(A)=\HHH^1(A),
\end{equation}
and this is still a Lie subalgebra under the commutator bracket. Just like usual Hochschild cohomology, it is invariant under Morita equivalence.

	\begin{prop}
	\label{prop:morita}
	If the finite-dimensional algebras $A$ and $B$ are Morita equivalent, then there is an isomorphism of Lie algebras $\HHH^1_{\rad}(A) \cong \HHH^1_{\rad}(B)$.
	\end{prop}
	\begin{proof}
We can assume that $B=\End_A(P) \cong eAe$, for a progenerator $P=eA$. It is well known that $\HHH^1$ is Morita invariant, and an explicit isomorphism is constructed as follows: for $\delta \in \Der(A)$, denote by
\begin{equation}
\delta^*:P \to P: x \mapsto \sum_i p_i\delta(f_i(x)),
\end{equation}
where $\{(p_i,f_i) \mid  p_i \in P, f_i \in \Hom_A(P,A)\}$ is a dual basis for $P$, that is, 
$x=\sum_i p_i f_i(x)$ for all $x\in P$. Then, as shown in \cite[Theorem 4.1]{farkas2002smooth}, the map
\begin{equation}
\Phi:\HHH^1(A) \to \HHH^1(B): \delta \mapsto [\delta^*,-]
\end{equation}
is well-defined, and is an isomorphism. Assume $\delta \in \Der_{\rad}(A)$, and $f \in \rad(\End_A(P))$. So there is an $r\in \rad(eAe) \subseteq \rad(A)$ such that $f(x)=r\cdot x$ for any $x\in P$. Then 
\begin{align}
[\delta^*,f](e)&=\sum_i p_i\delta(f_i(r)) - f\big(\sum_i p_i\delta(f_i(e))\big) \\
&= \sum_i p_i\delta(f_i(r)) - \sum_i  rp_i\delta(f_i(e)) \in \rad(P),
\end{align}
since $\delta \in \Der_{\rad}(A)$ and $f_i(r) \in \rad(A)$, so we are done.
	\end{proof}
	
	The following two lemmas will be crucial in the proof of our main theorem.

\begin{lem}
	\label{remark idempotents map to zero}
		Let $e_1,\ldots,e_n$ denote a full set of orthogonal primitive idempotents in $A$.
		If $\delta \in \Der(A)$, then there exists an $a\in \rad(A)$ such that 
		$\delta(e_i)=[a,e_i]$ for all $i\in\{1,\ldots,n\}$. 
	\end{lem}
	\begin{proof}
		Set $a=\sum_{j=1}^n (1-e_j)\delta(e_j) e_j$. Then
		
			\begin{align*}
			[a,e_i] &= (1-e_i)\cdot \delta(e_i)\cdot e_i-\sum_{j\neq i} e_i \cdot\delta(e_j)\cdot e_j
			\\&=-\delta(1-e_i)\cdot e_i+\sum_{j\neq i} \delta(e_i)\cdot e_j\\&=\delta(e_i)\cdot e_i + \delta(e_i)\cdot(1-e_i)=\delta(e_i)  \qedhere
			\end{align*}
		
	\end{proof}
	
		\begin{lem}
	\label{lem:idempotent-hh}
		If $e$ is an idempotent in a finite-dimensional $k$-algebra $A$, then there is a well-defined
		restriction map
		$$
			\res:\ \HHH^1_{\rad}(A) \longrightarrow \HHH^1_{\rad}(eAe):\ \delta + \Inn(A) \mapsto \delta|_{eAe} + \Inn(eAe) 
		$$
		where we assume without loss of generality that  $\delta(e)=0$.
	\end{lem}
	\begin{proof}
		By Lemma~\ref{remark idempotents map to zero} we can write $\HHH^1_{\rad}(A)$ as the quotient of 
		derivations mapping $e$ to zero by inner derivations with the same property. Moreover, an inner derivation $[a,-]$ mapping $e$ to zero coincides with $[eae,-]$ on $eAe$. Hence $\res$ is well-defined, and it is also a homomorphism of Lie algebras.
	\end{proof}
	The above would of course be equally true with the subscript ``$\rad$'' removed.
	We now give a useful sufficient condition for solvability of the Lie algebra $\HHH^1_{\rad}(A)$.

\begin{prop}
\label{prop:def-pin}
		Let $A$ be a finite-dimensional $k$-algebra, and $3 \leq n\in \NN$.
		There is a morphism of Lie algebras
		$$
			\pi_n: \ \Der_{\rad}(A/\rad (A)^n) \longrightarrow  \Der_{\rad}(A/\rad (A)^{n-1})
		$$
		and $\ker(\pi_n)$ is abelian. In particular, if $\Der_{\rad}(A/\rad (A)^{n-1})$ is solvable, then so is 
		$\Der_{\rad}(A/\rad (A)^n)$.
	\end{prop}
		\begin{proof}
		Let $\delta\in \Der_{\rad}(A/\rad (A)^n)$. Since $\delta(\rad (A))\subset \rad (A)$, also $\delta(\rad (A)^{n-1})\subseteq \rad (A)^{n-1}$. Hence we may define $\pi_n(\delta)(x+\rad (A)^{n-1})=\delta(x)+\rad (A)^{n-1}$. Now let $\xi$ be an element of the kernel of 
		$\pi_n$. That is, $\xi(x)\in \rad (A)^{n-1}$ for all $x\in A$. For $x,y\in \rad (A)$ we then have
		$\xi(x\cdot y)= x\cdot \xi(y)+\xi(x)\cdot y \in \rad (A)\cdot \rad (A)^{n-1}+\rad (A)^{n-1}\cdot \rad (A) = \rad (A)^n$. It follows that $\xi(\rad (A)^2) = \{0\}$. Hence, if $\mu$ is another element of $\ker(\pi_n)$, then 
		$\mu\circ \xi$ and $\xi\circ \mu$ are both zero (as $\mu$ and $\xi$ both map $A$ into $\rad (A)^{n-1} \subseteq \rad (A)^2$). In particular, $[\xi,\mu]=0$, and it follows that  $\ker(\pi_n)$ is abelian. For the last claim, if $\Der_{\rad}(A/\rad (A)^{n-1})$ is solvable, then so is its sub Lie algebra $\im(\pi)$. It follows that $\Der_{\rad}(A/\rad (A)^n)$ is solvable, being an extension of the solvable Lie algebra $\im(\pi_n)$ by the abelian Lie algebra $\ker(\pi_n)$.
	\end{proof}
	
		\begin{cor}
		\label{cor:iteration}
		Let $A$ be a finite-dimensional $k$-algebra, and $3 \leq n\in \NN$.
		There is a morphism of Lie algebras 
		$$
			\HHH^1_{\rad}(A/\rad (A)^n) \longrightarrow \HHH^1_{\rad}(A/\rad (A)^{n-1})
		$$
		with abelian kernel.
		 If 
		$\HHH^1_{\rad}(A/\rad (A)^{n-1})$ is solvable, then so is $\HHH^1_{\rad}(A/\rad (A)^n)$. In particular, if $\HHH^1_{\rad}(A/\rad (A)^{2})$ is solvable, then so is $\HHH^1_{\rad}(A)$.
	\end{cor}
	\begin{proof} The first statement follows because the map $\pi_n$ from above maps inner derivations surjectively onto inner derivations. The second statement follows immediately, and the third statement follows since $\rad (A)^n=0$ for some $n\in \NN$.
	\end{proof}

\subsection{Derivations preserving the radical}

Most results in this paper are phrased in terms of $\HHH^1_{\rad}(A)$, so to apply them to the usual Hochschild cohomology, we would like to know when $\HHH^1_{\rad}(A)=\HHH^1(A)$.
 
For a finite-dimensional algebra $A$ defined over a field of characteristic $0$, it is well known that every derivation $\delta \in \Der(A)$ preserves the radical, i.e. $\delta(\rad (A)) \subset \rad (A)$ (see for example~\cite[Theorem 4.2]{hochschild1942semi}), so in this case it is immediate that $\HHH^1_{\rad}(A)=\HHH^1(A)$. For algebras defined over a field of positive characteristic, this does not however have to be the case.  

\begin{ex}\label{ex:witt}
Assume $\cha(k)=p>0$ and consider the algebra $A=k[x]/(x^p)$.
\begin{enumerate}
\item  Since $A$ is commutative, $\HHH^1(A)\cong\Der(k[x]/(x^p))$ is the Witt algebra. For $p>2$ the Witt algebra is not solvable.
\item Consider the derivation $\delta \in \Der(A)$ determined by $\delta(x)=1$. Then $\delta \notin \Der_{\rad}(A)$, hence $\HHH^1_{\rad}(A)\neq\HHH^1(A)$. This example generalises to group algebras of (non-trivial) finite $p$-groups.
\end{enumerate}
In particular $\HHH^1(A)$ is non-solvable in this example (even though $A$ is of finite representation type), but $\HHH^1_{\rad}(A)$ is solvable by Corollary~\ref{cor:fin-rep-type} below.
\end{ex}

The following shows that, in spite of this example, there is a large class of algebras for which $\HHH^1$ and $\HHH^1_{\rad}$ coincide.

\begin{prop}
\label{prop:der-rad}
Assume that $A=kQ/I$. For every loop $a \in Q_1$, let $n_a$ denote the minimal integer such that $a^{n_a} \in \rad (A)^{n_a+1}$. If $p \nmid \prod_{a} n_a$, then 
$$\Der_{\rad}(A)=\Der(A) \quad\textrm{ and }\quad \HHH^1_{\rad}(A)=\HHH^1(A).$$
\end{prop}
\begin{proof}
Let $\delta \in \Der(A)$.
For $i\in Q_0$ denote by $e_i$ the corresponding standard idempotent.
By Lemma~\ref{remark idempotents map to zero} (as well as the fact that  $\Inn(A) \subset \Der_{\rad}(A)$) we can assume that $\delta(e_i)=0$ for all $i$.
Given an arrow $a \in Q_1$ with $s(a)=e_i$ and $t(a)=e_j$ which is not a loop, so $i \neq j$, we find that $\delta(a) \subset e_iAe_j \subset \rad (A)$.
Given a loop $a \in e_iAe_i$
 we have that $\delta(a)=\lambda e_i + r$, for $r \in \rad (A)$ and $\lambda\in k$. Then, since $\delta(\rad (A)^{n_a+1}) \subset \rad (A)^{n_a}$, we find
\begin{equation}
\rad (A)^{n_a}  \ni \delta(a^{n_a})=n_a\lambda a^{n_a-1} +\sum_{k=0}^{n_a-1}a^kra^{n_a-1-k},
\end{equation}
and since $p \nmid n_a$, we find that $\lambda a^{n_a-1} \in \rad (A)^{n_a}$. This contradicts the minimality of $n_a$, unless $\lambda=0$, in which case we have $\delta(a) \in \rad (A)$.
\end{proof}

\begin{rem}
	Proposition~\ref{prop:der-rad} is far from exhaustive. If, for example, for every loop in $a\in Q_1$ there is an arrow $b\in Q_1$ which is not a loop such that $t(a)=s(b)$ and $ab\in I$ (or $t(b)=s(a)$ and $ba\in I$), then $\HHH^1(A)=\HHH^1_{\rad}(A)$.
	
	To see this assume that $a\in Q_1$ is a loop and $e$ is the idempotent corresponding to $s(a)=t(a)$.  If there is 
	a derivation $\delta$ of $A$ with $\delta(e)=0$ such that $\delta(a)\not\in \rad(A)$, then $\delta(a)$ is a unit in $eAe$, and  $0=\delta(ab)=\delta(a)b+a\delta(b)$ then implies that $b=-\delta(a)^{-1}a\delta(b)$, that  is, $b\in \rad^2(A)$ (which contradicts $b$ being an arrow). This can be used to show that $\HHH^1(A)=\HHH^1_{\rad}(A)$ if $A$ is a non-local Brauer graph algebra (cf. \cite[Proof of Theorem~4.3]{schrolletal2019hochschild}).
\end{rem}

\begin{rem}
	Let $A=kQ/I$ be a finite-dimensional algebra, and let $e_i$ denote  the standard idempotent associated with the vertex $i\in Q_0$. One can define the ideal $$J=\sum_{i\neq j} Ae_iAe_jA \subset \rad(A).$$
	Derivations mapping the $e_i$'s to zero clearly stabilise $J$. Hence there is a morphism of Lie algebras
	$$
		\HHH^1(A) \longrightarrow \HHH^1(A/J)
	$$ 
	whose kernel is contained in $\HHH^1_{\rad}(A)$. The quiver of the algebra $A/J$ is a disjoint union of bouquets of loops, and if $A$ is of finite or tame representation type, then one can show that the number of loops in such a bouquet is bounded by one or two, respectively.
	This gives us the upper part of diagram~\eqref{diagram big picture}.
	
	 In concrete examples one could attempt to describe $\HHH^1(A/J)$ explicitly (the outcome should be similar to Example~\ref{ex:witt}), and then infer properties of $\HHH^1(A)$ from properties of $\HHH^1(A/J)$ and $\HHH^1_{\rad}(A)$. This is, however, not an approach we will pursue further in the present article.
\end{rem}

\section{Hochschild cohomology of radical square zero algebras}
\label{sec:rad-sq-zero}
By Corollary \ref{cor:iteration}, solvability of $\HHH^1_{\rad}$ can be established by looking at a radical square zero algebra. So let us assume in this section  that $A$ is a finite-dimensional $k$-algebra with $\rad (A)^2=\{0\}$, and $A/\rad(A)$ separable over $k$. By \cite[X.2.4]{auslander1997representation}, there is a stable equivalence of additive categories
\begin{equation}
F:\smodu(A) \to \smodu(\Sigma),
\end{equation}
where 
\begin{equation}
\Sigma=
\begin{pmatrix}
A/\rad (A) & \rad (A)\\
0 & A/\rad (A)
\end{pmatrix}
\end{equation}
is a hereditary algebra. 

\begin{lem}
\label{lem:der-sigma}
Assume $B$ is a basic hereditary algebra with $B/\rad (B)$ separable over $k$. Then for any $\delta \in \Der_{\rad}(B)$, we have that $\im(\delta)\subset \rad (B)$.
\end{lem}
\begin{proof}
Since a derivation $\delta\in \Der_{\rad}(B)$ maps $\rad (B)$ into itself, there is a homomorphism of Lie algebras 
		\begin{equation}
		\label{eq:der}
		\Der_{\rad}(B)\longrightarrow \Der_{\rad}(B/\rad (B)):\delta \mapsto \delta.
		\end{equation}
Since $B/\rad (B)$ is separable over $k$, and $B$ is basic, $\Der_{\rad}(B/\rad (B))=\HHH^1_{\rad}(B/\rad (B))=0$ and we are done.
\end{proof}

\begin{prop}
\label{prop:ders-rad}
		Assume that $A$ is basic. 
		Then there is an injective homomorphism of Lie algebras
		\begin{equation}
		\label{eq:inj}
		\Der_{\rad}(A)/K_A\longrightarrow \Der_{\rad}(\Sigma)
		\end{equation}
		and a surjective homomorphism of Lie algebras 
		\begin{equation}
		\label{eq:surj}
		\Der_{\rad}(A) \longrightarrow \Der_{\rad}(\Sigma)/K_\Sigma
		\end{equation}
		for the Lie ideals $K_A=\{\varphi \in \Der_{\rad}(A) \mid \varphi(A) \subset \rad (A) \text{ and } \varphi(\rad (A))=\{0\}\}$, and similarly for $K_\Sigma$.
	\end{prop}
	\begin{proof}
		Since a derivation $\delta\in \Der_{\rad}(A)$ maps $\rad (A)$ into itself, we get a well-defined map 
		$$
			\tilde \delta:\ \Sigma\longrightarrow \Sigma:\ \left[\begin{array}{cc} x+\rad (A) & r \\ 0& y+\rad (A) \end{array}\right]
			\mapsto 
			\left[\begin{array}{cc} \delta(x)+\rad (A) & \delta(r) \\ 0& \delta(y)+\rad (A) \end{array}\right]
		$$
		which one checks is a derivation. Moreover the assignment
\begin{equation}
\label{eq:assignment}
\varphi_A:\ \Der_{\rad}(A) \longrightarrow \Der_{\rad}(\Sigma): \delta \mapsto \tilde \delta
\end{equation}
		 is a homomorphism of Lie algebras. 
		A derivation $\delta$ lies in the kernel of $\varphi_A$ if and only if $\delta(\rad (A))=\{0\}$ and 
		$\delta(A) \subseteq \rad (A)$, i.e. if $\delta \in K_A$. 
		
		For the second part, choose a subalgebra
		$A_0 \cong A/\rad (A)$ of $A$ such that $A=A_0\oplus \rad (A)$. Then we get a corresponding embedding of algebras
		$$
		i: A=A_0\oplus \rad (A) \longrightarrow \Sigma: (x,r) \mapsto  	\left[\begin{array}{cc} x+\rad (A)& r \\ 0& x+\rad (A) \end{array}\right]
		$$ 
		and by Lemma \ref{lem:der-sigma} applied to $\Sigma$, an arbitrary derivation $\xi\in \Der_{\rad}(\Sigma)$ induces a derivation of $A$ with respect to this embedding, i.e. there is a commuting square
		\begin{equation}
		\begin{tikzcd}
		\Sigma \ar{r}{\xi} & \Sigma \\
		A \ar{u}{i} \ar[dashed]{r}{\delta} & A \ar{u}{i}
		\end{tikzcd}
		\end{equation}
		with $\delta \in \Der_{\rad}(A)$. Moreover, $\varphi_A(\delta)$ and $\xi$ agree on $\rad (\Sigma)$, so $\xi-\varphi_A(\delta) \in K_\Sigma$. This means that the composition of $\varphi_A$ with the map $\Der_{\rad}(\Sigma)\longrightarrow \Der_{\rad}(\Sigma)/K_\Sigma$ is surjective.
	\end{proof}

	\begin{cor}
	\label{cor:solv-rad}
		 Assume that $A$ is basic.
		 Then there is a surjective homomorphism of Lie algebras
		$$
			\varphi_A:\ \HHH^1_{\rad}(A) \twoheadrightarrow \HHH^1_{\rad}(\Sigma)
		$$
		with solvable kernel. In particular,
		 $\HHH^1_{\rad}(\Sigma)$ is solvable if and only if $\HHH^1_{\rad}(A)$ is solvable.
	\end{cor}
	\begin{proof}
		One checks that the map \eqref{eq:assignment} maps $\Inn(A)$ into $\Inn(\Sigma)$ (as $[a,-]$ gets mapped to $[{\rm diag}(a,a),-]$), and hence induces a map on $\HHH^1_{\rad}$.
		The derivations in $K_\Sigma$ are inner by Lemma~\ref{remark idempotents map to zero}, hence the first statement follows from \eqref{eq:surj} in Proposition~\ref{prop:ders-rad}.
		
		Since $A$ is basic, so is $\Sigma$. Hence the inner derivations on $\Sigma$ are solvable, and since the derivations in $K_A$ are inner (again by Lemma~\ref{remark idempotents map to zero}) it follows from \eqref{eq:inj} in Proposition~\ref{prop:ders-rad} that the kernel of $\varphi_A$ is solvable.
	\end{proof}
	
\begin{cor}
\label{cor:criterionsolv}
Let $A$ be a finite-dimensional $k$-algebra, and denote by $\Sigma$ the hereditary algebra corresponding to $A/\rad (A)^2$.
 If $\HHH^1_{\rad}(\Sigma)$ is solvable, then so is $\HHH^1_{\rad}(A)$.
\end{cor}	
\begin{proof}
This now follows by combining Corollary \ref{cor:solv-rad} with Corollary \ref{cor:iteration}  and Proposition~\ref{prop:morita}.
\end{proof}

Note that at this point, we have established all ingredients for diagram~\eqref{diagram big picture}.

\section{Hochschild cohomology of non-wild algebras}\label{sec:hoco non-wild}

From now on we will assume $k=\overline{k}$ is an algebraically closed field. For a quiver $Q$, we will denote by $\overline{Q}$ its underlying graph. For a finite-dimensional algebra $A=kQ/I$, we can consider $A/\rad(A)^2$ and the corresponding hereditary algebra $\Sigma$ (defined in Section \ref{sec:rad-sq-zero}). Then the basic algebra Morita equivalent to $\Sigma$ can be identified with $kQ^s$, where $Q^s$ is the separated quiver of $A$ (see \cite[X.2]{auslander1997representation}). The arrows of $Q$ and $Q^s$ correspond canonically, and for an arrow $a$ in $Q$ we write $a^s$ for the corresponding arrow in $Q^s$. 

\begin{prop}
\label{prop:rad-square}
Assume $A$ is a finite-dimensional $k$-algebra with $\rad(A)^2=\{0\}$, and corresponding separated quiver $Q^s$. Then:
\begin{enumerate}
\item \label{eq:finite} $A$ is of finite representation type if and only if $\overline{Q^s}$ is a disjoint union of Dynkin quivers.
\item \label{eq:tame} $A$ is of tame representation type if and only if $\overline{Q^s}$ is a disjoint union of Dynkin quivers and at least one Euclidean quiver.
\end{enumerate} 
\end{prop}
\begin{proof}
Since stable equivalence preserves representation type (see for example~\cite{krause1997stable}), $A$ is of finite (respectively tame) representation type if and only if $kQ^s$ is. Since $kQ^s$ is hereditary, it is of finite representation type if and only if the underlying graph of $Q^s$ is a disjoint union of Dynkin quivers~\cite[VII.5.10]{assem2006elements}, and of tame representation type if and only if the underlying graph of $Q^s$ is a disjoint union of Dynkin and Euclidean quivers~\cite[XIX.3.15]{simson2006elements}.
\end{proof}

We will now refine Corollary \ref{cor:criterionsolv} for non-wild algebras.

\begin{lem}
\label{lem:acyclic}
Let $Q$ denote an acyclic and connected quiver. Then $\HHH^1(kQ)=\HHH^1_{\rad}(kQ)=0$ if and only if $\overline{Q}$ is a tree.
\end{lem}
\begin{proof}
Note that since $Q$ is acyclic, Proposition \ref{prop:der-rad} ensures that $\Der(kQ)=\Der_{\rad}(kQ)$ and hence $\HHH^1(kQ)=\HHH^1_{\rad}(kQ)$. Since $A=kQ$ is Koszul, the Koszul complex 
\begin{equation}
K(A): 0 \to A \otimes_{kQ_0} kQ_1 \otimes_{kQ_0} A \xrightarrow{\nu} A \otimes_{kQ_0} A \xrightarrow{\mu} A \to 0
\end{equation}
is a free $A^e$-resolution of $A$, where $\mu$ denotes the multiplication map and
\begin{equation}
\nu(1 \otimes a \otimes 1)=a \otimes 1 - 1 \otimes a,
\end{equation}
for $a \in Q_1$. The dimension of $\HHH^1(kQ)$ is the Euler characteristic of the complex obtained by applying $\Hom_{A^e}(-,A)$ to $K(A)$, which can be seen to be
\begin{equation}
\dim_k(\HHH^1(kQ))=1-\#Q_0 +\sum_{a \in Q_1} \dim_k(e_{s(a)}(kQ)e_{t(a)}),
\end{equation}
so $\HHH^1(kQ)=0$ if and only if $\overline{Q}$ is a tree.
\end{proof}

\begin{prop}
\label{prop:solv-dynkin}
Let $Q$ denote an acyclic quiver such that its underlying graph $\overline{Q}$ is a disjoint union of Dynkin and Euclidean quivers. 
\begin{enumerate}
\item \label{eq:char-2}If $\cha(k)=2$, then $\HHH^1(kQ)=\HHH^1_{\rad}(kQ)$ is solvable.
\item \label{eq:char-not-2}If $\cha(k) \neq 2$, then $\HHH^1(kQ)=\HHH^1_{\rad}(kQ)$ is solvable if and only if $\overline{Q}$ does not contain an $\widetilde{A}_1$.
\end{enumerate} 
\end{prop}
\begin{proof}
If $Q=\sqcup_{i=1}^n Q_i$, then $\HHH^1(kQ)\cong \bigoplus_{i=1}^n \HHH^1(kQ_i)$ since Hochschild cohomology commutes with direct sums. The only Dynkin or Euclidean quivers which are not trees are the $\widetilde{A_n}$,  so \eqref{eq:char-2} and \eqref{eq:char-not-2} follow from Lemma \ref{lem:acyclic}, from the fact that $\HHH^1(kQ)\cong k$ is solvable for $\overline{Q}=\widetilde{A_n}$ and $n>1$, and from the fact that for $\overline{Q}=\widetilde{A}_1$, $\HHH^1(kQ) \cong \mathfrak{sl}_2$ is solvable if and only if $\cha(k)=2$.
\end{proof}

\begin{thm}
\label{thm:main}
Assume $A$ is a finite-dimensional $k$-algebra of non-wild representation type.
\begin{enumerate}
\item \label{eq:main-1}If $\cha(k)=2$, then $\HHH^1_{\rad}(A)$ is solvable.
\item \label{eq:main-2}If $\cha(k) \neq 2$, $\HHH^1_{\rad}(A)$ is solvable if $\overline{Q^s}$ does not contain an $\widetilde{A}_1$.
\end{enumerate}
\end{thm}
\begin{proof}
This now follows by combining Corollary \ref{cor:criterionsolv}, Proposition \ref{prop:rad-square}, and Proposition \ref{prop:solv-dynkin}.
\end{proof}

The condition in \eqref{eq:main-2} is not sufficient, as the following example illustrates.

\begin{ex}
Let $A=kQ/I$, where 
\begin{equation}
Q:
\begin{tikzcd}
1 \ar[shift left]{r}{a} \ar[shift right,swap]{r}{b} & 2 \ar{r}{c} & 3
\end{tikzcd}
\end{equation}
and $I=(ac)$. Then $A$ is special biserial, and hence tame. One computes that $\dim_k(\HHH^1(A))=2$, and hence $\HHH^1(A)$ is solvable. However, $\overline{Q^s}=\widetilde{A}_1 \sqcup A_2$.
\end{ex}

We will remedy this in the next section; for now we deduce the following two corollaries.

\begin{cor}
\label{cor:fin-rep-type}
Assume $A$ is a finite-dimensional $k$-algebra of finite representation type, then $\HHH^1_{\rad}(A)$ is solvable.
\end{cor}
\begin{proof}
This follows from Theorem \ref{thm:main} since $\widetilde{A}_1$ is not a Dynkin quiver. 
\end{proof}

\begin{cor}\label{cor:char2}
	If $\cha(k)=2$ and $A$ is an indecomposable non-local finite-dimensional $k$-algebra of non-wild representation type, then $\HHH^1(A)$ is solvable. 
\end{cor}
\begin{proof}
	We may assume without loss that $A=kQ/I$, and let $e_1,\ldots, e_n$ denote a full set of orthogonal primitive idempotents. Define $J=\sum_{i\neq j} A e_i A e_j A$. We know by Theorem~\ref{thm:main} that $\HHH^1_{\rad}(A)$ is solvable, and by considering diagram~\eqref{diagram big picture} it is clear that if $\HHH^1(A/J)$ is solvable, then so is 
	$\HHH^1(A)$.
	
	 Since $A$ is non-local, indecomposable and non-wild, every vertex in $Q$ has at most one loop attached to it. The reason is that if a vertex had at least two loops attached to it, then there would have to be at least one other arrow emanating from or terminating at that vertex (as $Q$ is connected and has more than one vertex by assumption). But then the separated quiver $Q^s$ is not a union of Dynkin and Euclidean quivers, a contradiction.  
	 
	 Hence the algebra $A/J$ is a direct product of algebras generated by a single element, that is, algebras of the form $k[x]/(x^n)$ for certain $n\in\NN$. Derivations of such an algebra are linear combinations of derivations $\delta_i$ defined by $\delta_i(x)=x^i$,  for $i\in \{1,\ldots, n-1\}$ if $n$ is odd and $i\in \{0,1,\ldots, n-1\}$ if $n$ is even. We have
	 $[\delta_i,\delta_j]=(j-i)\cdot \delta_{i+j-1}$, which implies that the commutator of two derivations is a linear combination of $\delta_i$'s with $i$ even. But if $i$ and $j$ are even, then $[\delta_i,\delta_j]=0$, which implies that the Lie algebra of derivations of $k[x]/(x^n)$ has derived length at most two.
\end{proof}

In particular, Corollary~\ref{cor:fin-rep-type} shows that in characteristic zero, the outer automorphism group of an algebra of finite representation type is solvable. As the following example shows, Corollary \ref{cor:fin-rep-type} is in some sense sharp.

\begin{ex}
For $A=k[x]/(x^n)$, and $p=\cha(k) \nmid n$, we have $\HHH^1_{\rad}(A)=\HHH^1(A)=\Der(A)=\langle \delta_1, \ldots, \delta_{n-1} \rangle$, where $\delta_i(x)=x^i$, and 
\begin{equation}
[\delta_i,\delta_{j}]=(j-i)\cdot \delta_{i+j-1},
\end{equation}
with the convention that $\delta_{i+j-1}=0$ for $i+j >n$. Hence $\HHH^1(A)$ is solvable but not nilpotent for $n>2$.
\end{ex}

\section{Algebras with Kronecker chains}

In this section we assume that $\cha(k)\neq 2$ (besides $k$ being algebraically closed). To obtain a more complete version of Theorem \ref{thm:main}, we will need to consider non-wild algebras with an $\widetilde{A}_1$ contained in $\overline{Q^s}$ separately. Our main result is Theorem \ref{thm:mainmain} below, which gives a complete characterisation of when the (radical preserving) first Hochschild cohomology of a non-wild algebra is solvable. To state this result, we need a few intermediate definitions.

\begin{defn}
\label{def:maxkron}
	Let $A=kQ/I$ be a finite-dimensional algebra.
	\begin{enumerate}
	\item We call a list  $C=((a_1,b_1),\ldots,(a_n,b_n))$ of pairs of arrows in $Q$ a 
	\emph{Kronecker chain} in $A$ if the following conditions are satisfied: 
	\begin{enumerate}
		\item $s(a_i)=s(b_i)$ and $t(a_i)=t(b_i)$ for all $i$.
		\item $t(a_i)=s(a_{i+1})$ for all $1\leq i < n$.
		\item\label{Kc:C} For all $1\leq i < n$  at least one of $a_ia_{i+1}, a_ib_{i+1}$, $b_ia_{i+1}, b_ib_{i+1}$
		is not in $I$.
		\item\label{Kc:D} All arrows $a_i,b_i$	involved in the chain are distinct.
	\end{enumerate}
	\item 
	We call $C$ a \emph{maximal Kronecker chain} if $C$ is not properly contained in another Kronecker chain. 
	\end{enumerate}
\end{defn}

It is easy to describe the shape of maximal Kronecker chains in non-wild finite-dimensional algebras.

\begin{lem} 
\label{rem:shape}
In a non-wild finite-dimensional algebra $A=kQ/I$, for a maximal Kronecker chain $C=((a_1,b_1),\ldots,(a_n,b_n))$, either
\begin{enumerate}
\item \label{eq:shape1}$n=1$ and $s(a_1)=t(a_1)$, or
\item \label{eq:shape2}$n>1, s(a_1)=t(a_n)$, and $s(a_i) \neq s(a_j)$, for $1\leq i \neq j \leq n$, or
\item \label{eq:shape3}$n >1, s(a_1)\neq t(a_n)$ and $s(a_i) \neq s(a_j)$ as well as $t(a_i) \neq t(a_j)$, for $1\leq i \neq j \leq n$.
\end{enumerate}
 In the first two cases the chain is a connected component of $Q$. In the third case there are no other arrows in $Q$ (other than the ones involved in $C$) whose source, respectively target, coincides with the source, respectively target, of any of the $a_i$.
\end{lem}
\begin{proof}
If a maximal Kronecker chain in an algebra $A$ is not of the type described in the statement, then the separated quiver $Q^s$ of $A$ contains subquivers which are not Dynkin or Euclidean quivers. Since $A$ is tame or of finite representation type, this cannot happen by Proposition \ref{prop:rad-square}. The second statement is similar.
\end{proof}

A Kronecker chain is simply a chain of Kronecker quivers embedded in $Q$,  and Lemma \ref{rem:shape} says that for a non-wild algebra, the chain is either a double loop $L_2$, an $\widetilde A_n$ with doubled arrows or an $A_n$ with doubled arrows. Condition \eqref{Kc:C} may seem unmotivated at first glance. As it turns out, a maximal Kronecker chain can potentially contribute an $\mathfrak{sl}_2$ to the first Hochschild cohomology of $A$, and condition \eqref{Kc:C} is essential for correct counting of these non-solvable constituents. To single out the chains that actually do contribute, and formulate our main theorem, we need two more definitions.

\begin{defn}\label{def iso by base change}
	Let $kQ/I$ and $kQ/J$ be finite-dimensional algebras. We say that $kQ/I$ and $kQ/J$ are \emph{isomorphic by base change} if there is an automorphism of $kQ$ which fixes the standard idempotents, induces a linear map on the vector space spanned by all arrows, and maps $I$ to $J$.
\end{defn}

\begin{defn}
\label{def:standardrels}
A maximal Kronecker chain $C=((a_1,b_1),\ldots,(a_n,b_n))$ in $A=kQ/I$ has \textit{standard relations} if $A$ is isomorphic by base change to an algebra $kQ/J$  in which the following relations hold:
		\begin{enumerate}[label=(S\arabic*)]
			\item\label{cond:S1} For any arrow $c$ in $C$ and any arrow $d$ not in $C$ we have $cd=dc=0$.
			\item\label{cond:S2} $a_ia_{i+1}=0$, $b_ib_{i+1}=0$ and $a_ib_{i+1}+b_ia_{i+1}=0$ for all $1\leq i <n$.
			\item\label{cond:S3} $a_na_1=0$, $b_nb_1=0$ and $a_nb_1+b_na_1=0$ (non-trivial only if $s(a_1)=t(a_n)$).
		\end{enumerate}
\end{defn}

Strictly for the purposes of proper counting we need to declare all rotated versions of a cyclic Kronecker chain equivalent.
\begin{defn}
	If $C=((a_1,b_1),\ldots,(a_n,b_n))$ and $D=((a_2,b_2),\ldots,(a_n,b_n), (a_1,b_1))$ are both 
	maximal Kronecker chains in $A=kQ/I$, then we call them equivalent. Equivalence of arbitrary maximal Kronecker chains is defined as the transitive closure of this equivalence relation.
\end{defn}

\begin{thm}
\label{thm:mainmain}
Let $A=kQ/I$ be a non-wild finite-dimensional algebra. Then there is an isomorphism of Lie algebras
\begin{equation}
\HHH^1_{\rad}(A) \cong \mathfrak{sl}_2^{\oplus m} \oplus \mathfrak{r},
\end{equation}
with $\mathfrak{r}$ solvable, where $m$ is the number of equivalence classes of maximal Kronecker chains with standard relations in $A$.
\end{thm}

\begin{cor}
Let $A=kQ/I$ be a non-wild finite-dimensional algebra. Then $\HHH^1_{\rad}(A)$ is solvable if and only if $A$ has no maximal Kronecker chains with standard relations.
\end{cor}

\subsection{Structure of the proof}

To prove Theorem \ref{thm:mainmain}, we need to relate Kronecker chains with standard relations to Hochschild cohomology. Assume that the separated quiver $Q^s$ of $A=kQ/I$ has a connected component which is a Kronecker quiver, i.e. 
	$$
	\begin{tikzcd}
		K_2 = i \ar[shift left]{r}{a^s}\ar[shift right, swap]{r}{b^s} & j  \subseteq Q^s
	\end{tikzcd}
	$$
	We denote the corresponding arrows in $Q$ by $a$ and $b$. 
	For definiteness, we will fix an isomorphism $\HHH^1(kK_2)\cong \HHH^1_{\rad}(kK_2)\cong \mathfrak{sl}_2$ once and for all as follows: 
\begin{align}
H&:=a^s \cdot \mathbf{1}_{a^s} - b^s \cdot \mathbf{1}_{b^s} \\
E&:=a^s \cdot \mathbf{1}_{b^s} \\
F&:= b^s \cdot \mathbf{1}_{a^s}
\end{align}
where  $\mathbf{1}_x$ is the indicator function sending $x$ to $1$ and every other basis element to $0$.

\begin{prop}\label{prop: Delta}
Assume $A=kQ/I$ is as above.
	\begin{enumerate}
		\item There is a morphism of Lie algebras 
		$$
		\Delta=\Delta_A(a,b):\HHH^1_{\rad}(A) \to \mathfrak{sl}_2,
		$$
		determined by the arrows $a,b\in A$.\footnote{$\Delta$ depends on the presentation of $A$, and we only  define it when $a$ and $b$ are arrows.}
		\item\label{item: formula delta} If, for some derivation $\delta \in \Der_{\rad}(A)$ mapping all standard idempotents to zero,
		\begin{equation}
		\Delta(\delta)=xH+ yE + zF
		\end{equation}
		then $$\delta(a)=(w+x) \cdot a + z\cdot b + r(a)$$ and $$\delta(b)=y \cdot a + (w-x)\cdot b + r(b)$$ for  $r(a),r(b) \in \rad^2(A)$ and some $w\in k$.
		Note that this implies that $\Delta_A(a,b)$ depends only on $A$ and the elements $a, b \in A$. 
		\item If $A$ is non-wild, there is a commutative diagram
		$$
			\begin{tikzcd}[column sep=large]
			\HHH^1_{\rad}(A) \ar{d}{\res} \ar{r}{\Delta_A(a,b)}& \mathfrak{sl}_2 \\
			\HHH^1_{\rad}(eAe) \ar[swap]{ur}{\Delta_{eAe}(a,b)}
			\end{tikzcd}
		$$
		 for any sum $e$ of standard idempotents containing the idempotents in $A$ corresponding to the vertices $i$ and $j$ in $Q^s$. 
	\end{enumerate}
\end{prop}
\begin{proof}
Define $\Delta_A(a,b)$ as the composition
\begin{equation}
\HHH^1_{\rad}(A) \xrightarrow{\pi} \HHH^1_{\rad}(A/\rad^2(A)) \stackrel{\varphi}{\longrightarrow} \HHH^1(kQ^s) \stackrel{\res}{\longrightarrow} \HHH^1(kK_2)\cong \mathfrak{sl}_2,
\end{equation}
where $\pi$ is the composition of the maps $\pi_n$ defined in Proposition \ref{prop:def-pin}, $\varphi$ is the map defined in Corollary \ref{cor:solv-rad}, and $\res$ was defined in Lemma \ref{lem:idempotent-hh}. Since all of these maps are morphisms of Lie algebras, so is $\Delta$.
The second statement then follows from the definition of $\Delta$ and the choice of our fixed isomorphism $\HHH^1(kK_2)\cong \mathfrak{sl}_2$. 

For the third statement note that if the separated quiver of $A$ has a connected component which is a Kronecker quiver, then the separated quiver of $eAe$ has a connected component that contains a Kronecker quiver (and potentially additional arrows and vertices). Since $A$ is non-wild, also $eAe$ is non-wild by \cite[I.4.7(b)]{erdmann2006blocks}, so this component of its separated quiver must not actually be any bigger than the Kronecker quiver. So we can choose a presentation of $eAe$ in which $a$ and $b$ correspond to arrows, and the corresponding arrows in the separated quiver form a connected component isomorphic to $K_2$.
Hence $\Delta_{eAe}(a,b)$ is defined, and we can write down the diagram from the statement. The fact that this diagram commutes follows immediately from point \eqref{item: formula delta}.
\end{proof}

The map $\Delta_A(a,b)$ as defined above is only defined when $a$ and $b$ are arrows, that is, it depends on the chosen presentation of $A$ as $kQ/I$. We could extend this definition to allow arbitrary $a$ and $b$ subject to conditions. However, for the purposes of this paper, it seems more sensible to think of $\Delta_A(a,b)$ as actually depending on the chosen presentation, as we will ultimately use information on the $\Delta_A(a,b)$'s to infer properties of the presentation.  All we need to know is how $\Delta_A(a,b)$ transforms when we change presentation. 
We will state  the transformation law for the following special isomorphism by base change, which is all we will need later on.

\begin{defn}
	Let $Q$ be a quiver, and let $a,b$ be two different arrows in $Q$.
	Assume $s(a)=s(b)$, $t(a)=t(b)$ and assume that $a$ and $b$ are the only arrows from $s(a)$ to $t(a)$. Given a matrix 
	$$
	X=\left( \begin{array}{cc} x_1 & x_2 \\  x_3&x_4 \end{array} \right) \in \operatorname{GL}_2(k)
	$$
	we define an automorphism $T_X=T_X(a,b)$ by
	$$
	T_X:\ kQ \longrightarrow kQ:\ a  \mapsto x_1a+x_2 b, \ b \mapsto x_3a + x_4 b,\ c \mapsto c \textrm{ for any other arrow $c$} 
	$$
	and $T_X$ fixes all standard idempotents in $kQ$.
\end{defn}

\begin{prop}\label{prop:delta base change}
	Let $A=kQ/I$ be a finite-dimensional algebra, and let $a$ and $b$ be two arrows in $Q$ such that 
	$\Delta_A(a,b)$  is defined. Let $Y \in \operatorname{GL}_2(k)$, and set $J=T_Y^{-1}(I)$.
	Then
	$$
	\Delta_{kQ/J}(a,b)(T_Y^{-1}\circ \delta \circ T_Y) = Y \cdot \Delta_A(a,b)(\delta) \cdot Y^{-1} \quad\textrm{for all $\delta \in \Der_{\rad}(A)$}
	$$
	and 
	$$
	\Delta_{kQ/J}(c,d)(T_Y^{-1}\circ \delta \circ T_Y) = \Delta_A(c,d)(\delta)  \quad\textrm{for all $\delta \in \Der_{\rad}(A)$}		
	$$
	for any other two arrows $c$ and $d$ for which $\Delta_A(c,d)$ is defined.
\end{prop}
\begin{proof}
	By Proposition~\ref{prop: Delta} we know that for any $\delta \in \Der_{\rad}(A)$
	$$
	\left(\begin{array}{c} \delta(a)\\\delta(b) \end{array}\right) = (\Delta_A(a,b)(\delta) + w \cdot \id_2) \cdot 
	\left(\begin{array}{c} a\\\ b \end{array}\right) + \left(\begin{array}{c} r(a)\\r(b) \end{array}\right)
	$$
	for certain $w\in k$, $r(a),r(b)\in \rad^2(A)$. Conversely, it follows that $\delta(a)$ and $\delta(b)$ uniquely determine $\Delta_A(a,b)(\delta)$ as an element of $\mathfrak{sl}_2 \subset k^{2\times 2}$.
	
	We can compute
	$$
	\left(\begin{array}{c} T_Y^{-1}(\delta(T_Y(a)))\\T_Y^{-1}(\delta(T_Y(b))) \end{array}\right) = Y\cdot (\Delta_A(a,b)(\delta) + w \cdot \id_2)\cdot Y^{-1} \cdot 
	\left(\begin{array}{c} a\\\ b \end{array}\right) + \left(\begin{array}{c} r(a)'\\r(b)' \end{array}\right)
	$$
	for certain $r(a)', r(b)'\in \rad^2(kQ/J)$ (note that we view $a$ and $b$ as elements of $kQ/J$ here), which uniquely determines $\Delta_{kQ/J}(a,b)(T_Y^{-1}\circ \delta \circ T_Y)$ and therefore shows the first assertion. The second assertion follows since, by definition, $T_Y(c)=c$ for any arrow $c$ which is neither $a$ nor $b$. 
\end{proof}

\begin{defn}
A maximal Kronecker chain $C=((a_1,b_1),\ldots,(a_n,b_n))$ is called \emph{surjective} if $\Delta_A(a_i,b_i)$ is surjective for some $i$.
\end{defn}
	
The following characterisation of surjective maximal Kronecker chains is the main ingredient for the proof of Theorem \ref{thm:mainmain}.

\begin{thm}\label{thm:kronecker main}
		Let $A=kQ/I$ be a non-wild finite-dimensional algebra. Given a maximal Kronecker chain $C=((a_1,b_1),\ldots,(a_n,b_n))$  the following hold:
	\begin{enumerate}
		\item\label{statement-1}$C$ is surjective if and only if $C$ has standard relations.
		\item\label{statement-2} If $C$ is surjective, then all $\Delta_A(a_i,b_i)$ are conjugate.
	\end{enumerate}
\end{thm}

The proof of Theorem~\ref{thm:kronecker main} is quite long, but can be split up into steps as follows:
\begin{enumerate}
\item The if-direction: see \S\ref{sec:if}.
\item The local case: see \S\ref{sec:local}. 
\item Reduction to $\rad^3(A)=0$: see \S\ref{sec:reduction}.
\item Establishing \ref{cond:S2}--\ref{cond:S3} locally: see \S\ref{sec:locals2s3}.
\item Establishing \ref{cond:S1}: see \S\ref{sec:s1}.
\item Establishing \ref{cond:S2}--\ref{cond:S3} globally: see \S\ref{sec:s2s3glo}.
\end{enumerate}
Finally, in \S\ref{sec:proofkromain} and \S\ref{sec:proofmainmain}, we give the proofs of Theorem \ref{thm:kronecker main} and Theorem \ref{thm:mainmain}.

\subsection{The if-direction}
\label{sec:if}
As described above, we start with the ``if''-direction of Theorem \ref{thm:kronecker main}\eqref{statement-1}. The statement below gives a bit more information, which we will need when proving Theorem \ref{thm:mainmain}.

\begin{prop}\label{prop:trivial if}
	Let $A=kQ/I$ be a non-wild finite-dimensional algebra, and let $C=((a_1,b_1),\ldots,(a_n,b_n))$ be a maximal Kronecker chain. If the relations \ref{cond:S1}--\ref{cond:S3} hold, then $C$ is surjective. 
	Moreover, there is a Lie subalgebra $\mathfrak{g} \subset \HHH^1_{\rad}(A)$ such that $\Delta_A(a_i,b_i)|_{\mathfrak{g}}:\mathfrak{g}  \to \mathfrak{sl}_2$ is an isomorphism, for every $1\leq i \leq n$.
\end{prop}
\begin{proof}
Given $x,y,z \in k$, we can define a derivation $\delta=\delta_{(x,y,z)}$ on $kQ$ by setting
	$$
		\delta(a_i) = x a_i+z b_i, \quad \delta(b_i)= y a_i -x b_i
	$$
	for each $1\leq i\leq n$ and $\delta(c)=0$ for any other arrow. We would have 
	$$\Delta_A(a_i,b_i)(\delta)=xH + yE + zF \in \mathfrak{sl}_2 \quad \textrm{for any $i$} $$
	provided $\delta$ actually induces a derivation on $A$, that is, provided $\delta(I) \subseteq I$.
	In that case $\Delta_A(a_i,b_i)$ would be surjective, and we can then also set $\mathfrak{g}=\langle \delta_{(1,0,0)}, \delta_{(0,1,0)}, \delta_{(0,0,1)} \rangle_k$. Hence we are reduced to showing $\delta(I)\subseteq I$.
	
	Consider the ideal $J_0$ generated by paths made up out of at least one arrow in $C$ and at least one arrow not in $C$. Since \ref{cond:S1} holds, $J_0 \subset I$, and by definition of $\delta$, we also have $\delta(J_0) \subset J_0$, so $\delta$ induces a derivation on $kQ/J_0$. The algebra $kQ/J_0$ is spanned by paths involving either exclusively arrows in $C$ or exclusively arrows not in $C$. Note that a path involving exclusively arrows in $C$ cannot have the same source and target as a path involving exclusively arrows not in $C$. The reason is that if there were any arrows sharing their source or target with one of the $a_i$ or $b_i$, the algebra $kQ^s$ would be wild, contradicting the assumption that $A$ is not wild. So actually $I/J_0$ is generated by linear combinations of paths involving exclusively arrows outside $C$ (these get mapped to zero by $\delta$) and linear combinations of paths involving exclusively arrows in $C$ (for which we still have to check that $\delta$ maps them back into $I$). 
	
	Define $J_1 = J_0 + (a_ia_{i+1}, b_ib_{i+1}, a_ib_{i+1}+ b_ia_{i+1}, a_na_1, b_nb_1, a_nb_1+b_na_1 \ | \ 1 \leq i < n)$.  Because of \ref{cond:S2} and \ref{cond:S3}, $J_1\subseteq I$, and one verifies that $\delta(J_1)\subseteq J_1$. Note that any path of length $\geq 3$ consisting of arrows in $C$ is already contained in $J_1$. 
	Moreover, if $e_i\in kQ$ denotes the idempotent corresponding to the source of the arrow $a_i$, and $e_{n+1}$ corresponds to the target of $a_n$, we have $\dim_k e_i(kQ/J_1)e_{i+2} =1$ for 
	$1\leq i <n$, and also $\dim_k e_n(kQ/J_1)e_{2}\leq 1$. Now we are reduced to showing that 
	$\delta$ maps $e_{i}(I/J_1)e_{i+2}$ into itself for each $i$ (and possibly the same for $e_n(I/J_1)e_2$). But since each of these is a subspace of an (at most) one-dimensional vector space stabilised by $\delta$, this is actually trivially the case.
\end{proof}

\subsection{The local case}
\label{sec:local}
If $A$ is connected local, then by Lemma \ref{rem:shape} the quiver of $A$ is a two-loop quiver $L_2$, so Theorem~\ref{thm:kronecker main} \eqref{statement-2} is an empty statement, and the following is all we need to show.

\begin{prop}
	\label{prop:local}
	Let $A=k\langle a,b\rangle/I$ be a non-wild local algebra such that 
	$\Delta_A(a,b)$ is surjective. Then $A$ is isomorphic by base change to $kQ/J$, where
	 $J=(a^2,b^2,ab,ba)$ or $J=(a^2, b^2, ab+ba)$.
\end{prop}
\begin{proof}
	Recall that we assume, by convention, that $I$ is contained in the paths of length two (otherwise $\Delta_A(a,b)$ would be undefined, as the quiver of $A$ would not actually contain a double loop). In particular $I \subseteq (a^2,b^2,ab,ba)$.
	
	We will  make use of a classification result for tame local algebras by Ringel. Namely, by 
	\cite[Theorem (1.4)]{ringel1975representation} we may assume that $A$ is isomorphic to $kQ/J$ for an ideal $J$ for which one of the following holds:
	\begin{enumerate}
		\item\label{case 1} $ab+r \in J$ and $ba+q_1 a^2 + q_2 b^2+s \in J$ for $q_1,q_2\in k$; or
		\item\label{case 2} $a^2+r\in J$ and $b^2+s\in J$
	\end{enumerate}
	where in either case $r$ and $s$ only involve monomials in $\rad^3(kQ/J)$. Note that in case (4) in \cite{ringel1975representation} one needs to apply the change of variables $\tilde Y=Y-X$ to get the presentation into the form given above, but in all other cases no adaptations are necessary.
	 In fact, we may assume that $A$ and $kQ/J$ are isomorphic by base change, and replace, without loss of generality, $I$ by $J$.
	By our assumption on the surjectivity of $\Delta_A(a,b)$ we may, by Proposition~\ref{prop: Delta}, also assume that there 
	are  $\delta_1, \delta_2 \in \Der(A)$ such that
	\begin{align}
	\delta_1(a) &= w_1 a + b + r_1(a), & \delta_2(a) &= w_2 a + r_2(a), \\
	\delta_1(b)&= w_1 b + r_1(b), & \delta_2(b)&= a + w_2b + r_2(b),
	\end{align}
	for some $w_1,w_2\in k$ and $r_1(a), r_2(a),r_1(b),r_2(b) \in \rad^2(A)$.
	
	In case~\eqref{case 1} we have $ab\in \rad^3(A)$ and therefore
	$$\rad^3(A)\ni \delta_1(ab)= b^2+ 2w_1ab + r_1(a)b+ar_1(b),$$ hence $b^2 \in \rad^3(A)$. Similarly
	$$\rad^3(A)\ni \delta_2(ab)= a^2 +2w_2ab+ r_2(a)b+ar_2(b),$$
	hence $a^2 \in \rad^3(A)$. We conclude that $a^2,b^2,ab, ba\in \rad^3(A)$, which shows that 
	$\rad^2(A)\subseteq \rad^3(A)$. By the Nakayama lemma it follows that $\rad^2(A)=\{0\}$, so $I=(a^2,b^2,ab,ba)$.
	
	In case~\eqref{case 2} we have $a^2\in \rad^3(A)$ and therefore
	$$
	\rad^3(A)\ni \delta_1(a^2) = ab+ba+ 2w_1a^2+r_1(a)a +ar_1(a),
	$$
	hence $ab+ba\in \rad^3(A)$, which implies skew-commutativity mod $\rad^3(A)$. It follows that 
	$a^3, a^2b, ab^2, b^3$ map to generators of the $k$-vector space $\rad^3(A)/\rad^4(A)$. But since $a^2,b^2\in\rad^3(A)$ we have  $a^3, a^2b, ab^2, b^3\in \rad^4(A)$, which
	then implies $\rad^3(A)=\rad^4(A)$. The Nakayama lemma now yields $\rad^3(A)=\{0\}$, which implies $I \supseteq (a^2, b^2, ab+ba)$. But $k\langle a,b\rangle / (a^2, b^2, ab+ba)$ has socle $ab$, so if $I$ properly contains $(a^2, b^2, ab+ba)$, then $I=(a^2,b^2,ab,ba)$. Either way our claim is true.
\end{proof}

\subsection{Reduction to $\rad^3(A)=0$}
\label{sec:reduction}
The next step is a reduction to the case of algebras whose radical cubes to zero.

\begin{lem}\label{lemma:red mod rad3}
	Let $A=kQ/I$ be a non-wild finite-dimensional algebra and let $C=((a_1,b_1),\ldots,(a_n,b_n))$ be a maximal Kronecker chain. If all of the following hold:
	\begin{enumerate}
		\item \label{eq:conddd-1}For any arrow $c$ in $C$ and any arrow $d$ not in $C$ we have $cd,dc\in \rad^3(A)$.
		\item \label{eq:conddd-2}$a_ia_{i+1}, b_ib_{i+1}, a_ib_{i+1}+b_ia_{i+1}\in \rad^3(A)$ for all $1\leq i <n$.
		\item \label{eq:conddd-3}If $s(a_1)=t(a_n)$, then $a_na_1, b_nb_1, a_nb_1+b_na_1\in \rad^3(A)$.
	\end{enumerate}
	then the relations \ref{cond:S1}--\ref{cond:S3} hold in $C$.
\end{lem}
\begin{proof}
	Condition \eqref{eq:conddd-2} immediately implies that 
	\begin{equation}
	\label{eq:paths-1}
	a_ia_{i+1}a_{i+2}, a_ia_{i+1}b_{i+2}, b_ib_{i+1}a_{i+2}, b_ib_{i+1}b_{i+2}, a_ib_{i+1}b_{i+2}, b_ia_{i+1}a_{i+2} \in \rad^4(A),
\end{equation}
for $1\leq i \leq  n-2$. Similarly,
\begin{equation}
\label{eq:paths-2}
a_ib_{i+1}a_{i+2} + b_ia_{i+1}a_{i+2}, b_ia_{i+1}b_{i+2} + a_ib_{i+1}b_{i+2}\in \rad^4(A),
\end{equation}
from which we deduce that all paths of length three involving only arrows in $C$ are actually contained in $\rad^4(A)$. The assumptions also imply that any path of length three involving at least one arrow in $C$ and at least one arrow not in $C$ is contained in $\rad^4(A)$. It follows that any path of length three involving at least one arrow in $C$ is contained in $\rad^4(A)$. 

	As $kQ^s$ needs to be non-wild, it follows by Lemma \ref{rem:shape} that there are no arrows other than $a_i$ and $b_i$ emanating from the vertex $s(a_i)=s(b_i)$. Hence, the fact that $a_i\rad^2(A)\subseteq \rad^4(A)$ (which we have just demonstrated) implies that $a_i \rad^2(A) \subseteq a_i \rad^3(A)+b_i\rad^3(A)$. The same is true with $a_i$ replaced by $b_i$. Hence $a_i\rad^2(A) + b_i\rad^2(A) \subseteq a_i\rad^3(A) + b_i\rad^3(A)$. Now the Nakayama lemma implies
	$a_i\rad^2(A) + b_i\rad^2(A)=\{0\}$, which shows that \ref{cond:S2} and \ref{cond:S3} follow from conditions \eqref{eq:conddd-2} and \eqref{eq:conddd-3}, as these elements are contained in $a_i\rad^2(A) + b_i\rad^2(A)$ and 
	$a_n\rad^2(A) + b_n\rad^2(A)$, respectively, rather than just $\rad^3(A)$.
	
	As for condition \ref{cond:S1}, note that as $kQ^s$ is non-wild, for every $i$, $a_i$ and $b_i$ are the only arrows emanating 
	from $s(a_i)=s(b_i)$, and the only arrows pointing to $t(a_i)=t(b_i)$. So condition \ref{cond:S1} really only concerns $i=1$ and $i=n$. If $t(a_n)=s(a_1)$, then $C$ is actually an entire connected component of $Q$, thus rendering condition \ref{cond:S1} vacuous. If $t(a_n)\neq s(a_1)$, then condition \ref{cond:S1} can be reformulated as saying $a_n\rad(A) =b_n \rad(A)=\rad(A)a_1=\rad(A)b_1= \{0\}$. By condition \eqref{eq:conddd-1} we have $a_n \rad(A)+b_n\rad(A) \subseteq a_n\rad^2(A)+b_n\rad^2(A)$, which implies 
	 $a_n \rad(A)+b_n\rad(A)=\{0\}$. In the same vein we have  $\rad(A)a_1+\rad(A)b_1 \subseteq \rad^2(A)a_1+\rad^2(A)b_1$, which implies 
	 $\rad(A)a_1+\rad(A)b_1=\{0\}$.
\end{proof}

\begin{cor}\label{cor: reduction to rad3 zero}
	Theorem~\ref{thm:kronecker main} holds for a non-wild finite-dimensional $k$-algebra $A=kQ/I$ if it holds for $A/\rad^3(A)$.
\end{cor}
\begin{proof}
	If a maximal Kronecker chain in $A$ is surjective then so is the corresponding maximal Kronecker chain in $A/\rad^3(A)$, as $\Delta_A(a,b)$  factors through $\Delta_{A/\rad^3(A)}(a,b)$ by definition (for any two arrows $a,b$ in $Q$ forming a Kronecker). In particular all $\Delta_A(a_i,b_i)$'s are conjugate if all $\Delta_{A/\rad^3(A)}(a_i,b_i)$'s are conjugate. 
	This takes care of the second part of Theorem~\ref{thm:kronecker main}.
	
	Now note that if $kQ/J_3$ for some admissible ideal $J_3$ is isomorphic by base change to $A/\rad^3(A)$, then, by taking the preimage of $I$ under the automorphism of $kQ$ inducing the isomomorphism between $kQ/J_3$ and $A/\rad^3(A)$, we find an ideal $J$ in $kQ$ such that $A$ is isomorphic by base change to $kQ/J$ and $J_3/J = \rad^3(kQ/J)$.
	
	 Hence, given a maximal Kronecker chain $C$ such that the relations \ref{cond:S1}--\ref{cond:S3} hold in an algebra isomorphic by base change to $A/\rad^3(A)$, it follows by Lemma~\ref{lemma:red mod rad3} 
	  that the relations \ref{cond:S1}--\ref{cond:S3}  hold in an algebra isomorphic by base change to $A$ as well. Conversely, if  \ref{cond:S1}--\ref{cond:S3}  hold in $A$ (up to isomorphism by base change) then they trivially also hold in $A/\rad^3(A)$.
	  
	  It remains to show that if a maximal Kronecker chain $C$ in $A/\rad^3(A)$ is surjective, then so is the corresponding Kronecker chain in $A$. However, the fact that $C$ is surjective in $A/\rad^3(A)$ implies that \ref{cond:S1}--\ref{cond:S3} hold in $A/\rad^3(A)$ up to isomorphism by base change, and therefore in $A$ (as we have seen). The fact that $C$ is surjective in $A$ now follows from Proposition~\ref{prop:trivial if}.
\end{proof}

\subsection{Establishing \ref{cond:S2}--\ref{cond:S3} locally}
\label{sec:locals2s3}
We can now assume $A=kQ/I$ is non-local and $\rad^3(A)=0$. In this section we will  establish the relations \ref{cond:S2}--\ref{cond:S3} ``locally'', that is, for (non-cyclic) Kronecker chains of length two. 
Lemma~\ref{lemma double kronecker} below is the key technical ingredient, and also the main culprit behind the failure of Theorem~\ref{thm:kronecker main} in characteristic two.

\begin{lem}\label{lemma double kronecker}
	Let $A=kQ/I$ be a non-wild finite-dimensional algebra, where 
	\begin{equation}
	Q:
	\begin{tikzcd}
	1 \ar[shift left]{r}{a} \ar[shift right,swap]{r}{b} & 2 \ar[shift left]{r}{c} \ar[shift right,swap]{r}{d} & 3 
	\end{tikzcd}
	\end{equation}
	and $I$ is an admissible ideal. Assume that either $\Delta_A(a,b)$ or $\Delta_A(c,d)$ is surjective. 
	Then $A$  is isomorphic by base change to $kQ/J$, where either $J=(ac, bd, ad+bc)$ or $J=(ad, bc, ac, bd)$.	
\end{lem}
\begin{proof}
	We will assume without loss of generality that $\Delta_A(a,b)$ is surjective. The case 
	where $\Delta_A(c,d)$ is surjective is analogous, and technically reduces to the case where  $\Delta_A(a,b)$ is surjective by passing to the opposite algebra (see \cite[Ch. XIX, Lemma 1.4]{simson2006elements}).
	
	As a first step, note that $\dim_k (e_1 A e_3) \leq 2$, since otherwise $(e_1+e_3)A(e_1+e_3)=kK_3$ or $kK_4$ would be wild. We consider several cases. In each case we will successively apply isomorphisms by base change and without loss of generality replace $A$ by the isomorphic algebra, until we either obtain that $A$ is as claimed, or at least $I \supseteq (ac, bd)$ (which is the case we will then deal with at the end).
	\begin{enumerate}
		\item
		Assume $y_1\cdot ac+y_2\cdot bc\in I$ for $(0,0)\neq (y_1,y_2)\in k^2$ and 
		$y_3\cdot ad+y_4\cdot bd\in I$ for $(0,0)\neq (y_3,y_4)\in k^2$. 
		\begin{enumerate}
			\item If $(y_1,y_2)$ and $(y_3,y_4)$ are linearly independent, then 
			set 
			$$
				Y = \left( \begin{array}{cc} y_1&y_2 \\ y_3&y_4 \end{array}\right)
			$$
			and define $J=T_Y(a,b)^{-1}(I)$.
			Then $J$ contains $ac$ and $bd$, and we may replace $I$ by $J$. 
			
			\item If $(y_1,y_2)=z\cdot (y_3,y_4)$ for some $z\in k$, then we may complete $(y_1,y_2)$
			by another row $(y_5, y_6)$ to form an invertible $2\times 2$-matrix, and then proceed as before to obtain that $A$ is isomorphic to $kQ/J$, where $J$ contains $ac$ and $ad$. Now replace $A$ by $kQ/J$. Since $\Delta_A(a,b)$ is surjective, by Proposition \ref{prop: Delta}, there is a derivation $\delta$ on $A$ such that $\delta(a)=wa+b$ for some $w\in k$ (by Lemma \ref{remark idempotents map to zero}, we can assume the $\rad^2(A)$ term is zero). By Lemma \ref{remark idempotents map to zero}, we can assume $\delta(d)=z_1c+z_2 d$ for certain $z_1,z_2\in k$ and therefore
	$$0=\delta(ad)=(wa+b)d + a(z_1c+z_2 d)=bd$$
		Hence $bd\in I$. Similarly we obtain $bc \in I$. Hence $I$ is equal to $(ac,ad,bc,bd)$. 
		\end{enumerate}
	
		\item Now assume that either $ac$ and $bc$ or $ad$ and $bd$ are linearly independent. By applying an automorphism we may assume that $ad$ and $bd$ are linearly independent, and therefore $I$ contains the entries of
		$$
		\left(\begin{array}{cccc}
		1&0&x_1&x_2 \\
		0&1&x_3&x_4
		\end{array}\right) \cdot (ac,bc,ad,bd)^\top = 
		\left(\begin{array}{c|c}
		{\mathrm{Id}}_2& X
		\end{array}\right) \cdot (ac,bc,ad,bd)^\top
		$$ 
		If $Y\in \operatorname{GL}_2(k)$ then one verifies that $J=T_Y^{-1}(I)$ contains the entries of
		$$
		\left(\begin{array}{c|c}
		{\mathrm{Id}}_2& YXY^{-1}
		\end{array}\right) \cdot (ac,bc,ad,bd)^\top
		$$
		and we may replace $I$ by $J$. That is, we may assume that $X$ is in Jordan normal form. Then either 
		$I\supseteq (ac+y_1ad, bc + y_2 bd)$ or $I\supseteq (ac+y_1ad, bc  +ad+ y_1 bd)$ for $y_1,y_2\in k$.
		
		\begin{enumerate}
			
			\item\label{eq:stepa}
			In the first case, if $y_1\neq y_2$ consider the algebra homomorphism which sends $a$ to $a$, $b$ to $b$, 
			$c$ to $c+y_1d$ and $d$ to $c+y_2 d$. The preimage $J$ of $I$ under this homomorphism contains $ac$ and $bd$, and we may replace $I$ by $J$. 
			
			Similarly, if $y_1=y_2$ we can replace $I$ by a $J$ which contains 
			$ac$ and $bc$. As $kQ/(ac,bc)$ is wild (since $kQ/(ac,bc,c) \cong kQ/(c)$ is wild) $I$ must also contain an element of the form 
			$x_1ad+x_2bd$ for $(0,0) \neq (x_1,x_2)\in k^2$. By applying an automorphism we may assume without loss of generality that 
			$I\supseteq (ac, bc, ad)$.
			As we assume $\Delta_A(a,b)$ surjective there is a derivation  $\delta$
			such that $\delta(a)=wa+b$ for some $w\in k$. Hence,  for certain $z_1,z_2\in k$,
			$$0=\delta(ad)=(wa+b)d + a(z_1c+z_2 d)=bd$$
			which implies that $bd\in I$, i.e. $I=(ac,bc,ad,bd)$.
			
			\item
			In the second case we can replace $I$ by $J\supseteq (ac,bc+ad)$.
			As we assume $\Delta_A(a,b)$ surjective there is a derivation  $\delta$
			such that $\delta(a)=wa+b$ and $\delta(b)=a+wb$ for some $w\in k$. Hence,  for certain $z_1,z_2,z_3,z_4\in k$,
			\begin{align}
			0&=\delta(ad+bc)\\
			&=(wa+b)d + a(z_3c+z_4 d)+ (a+wb)c +b(z_1c+z_2d)\\
			&=(z_2+1) bd + z_1bc+z_4ad \\
			&= (z_2+1) bd + (z_4-z_1) ad
			\end{align}
			and 
			$$
			0=\delta(ac) = (wa+b)c +a(z_1c+z_2d)= bc + z_2 ad=(z_2-1) ad 
			$$
			If $z_2\neq 1$ we get $I\supseteq (ac,bc,bd)$, and we have already seen in step \eqref{eq:stepa} that our assumptions then imply $I\supseteq (ac,ad,bc,bd)$. If $z_2=1$, then we get $bd-z\cdot ad=(b-za)d\in I$ for some $z\in k$ (and this is where the assumption $\operatorname{char} (k)\neq 2$ is crucial).  After applying an automorphism we may assume that $I$ contains  $ac$, $bd$ and $ad+bc$.
		\end{enumerate}
	\end{enumerate}

	In either case we may now assume that $I \supseteq (ac, bd)$. Since $\Delta_A(a,b)$ is surjective there is a derivation $\delta$ on $A$ such that $\delta(a)=wa+b$ for some $w\in k$. Then $\delta(c)=z_1c+z_2 d$ for certain $z_1,z_2\in k$ and therefore
$$0=\delta(ac)=(wa+b)c + a(z_1c+z_2 d)=bc+z_1ad$$
Hence $bc+z_1ad\in I$. In case $z_1\neq 0$ we may then apply an automorphism, and then we may assume $I \supseteq (ac, bd, ad+bc)$. If $z_1=0$ then we have seen already that $I=(ac,bc,ad,bd)$ follows.
\end{proof}

Note that if $J=(ac, bd, ad+bc)$ in the previous lemma, then  $\Delta_{kQ/J}(a,b)=\Delta_{kQ/J}(c,d)$, as the following proposition shows. It will be quite important later on that there is a kind of converse to this (see Proposition~\ref{prop: delta normalises relations}). 

\begin{prop}\label{prop: deltas are equal}
	Let $A=kQ/I$ be a non-wild finite-dimensional algebra, and let $a,b,c,d$ be distinct arrows in $Q$
	such that $s(a)=s(b)$, $t(a)=t(b)=s(c)=s(d)$ and $t(c)=t(d)$. If 
	$ac, bd, ad+bc \in I$, but $ad\not \in I$, then $\Delta_A(a,b)=\Delta_A(c,d)$.
\end{prop}
\begin{proof}
	Our assumption that $A$ is not wild implies that $a$ and $b$ are the only arrows in $Q$ from $s(a)$ to $t(a)$, and the analogous statement is true for $c$ and $d$. Hence, if $\delta$ is a derivation of $A$, then there are 
	$x_1,\ldots,x_4,y_1,\ldots,y_4 \in k$ such that 
	$$
	\delta(a)=x_1a+x_2b\quad \delta(b)=x_3a+x_4b \quad \delta(c)=y_1c+y_2d \quad \delta(d)=y_3c+y_4d
	$$
	Note that 
	\begin{equation}
	\Delta_A(a,b)(\delta)=\frac{1}{2}(x_1-x_4)H+x_3E+x_2F, \quad
	\Delta_A(c,d)(\delta)=\frac{1}{2}(y_1-y_4)H+y_3E+y_2F
	\end{equation}
	Now $0=\delta(ac)=x_2bc + y_2 ad=(x_2-y_2) bc$, which implies $x_2=y_2$. In the same vein, 
	$0=\delta(bd)=(x_3-y_3) ad$, which implies $x_3=y_3$. Lastly
	$$
	0=\delta(ad+bc)= (x_1+y_4) ad + (x_4+y_1) bc = (x_1-x_4 + y_4-y_1) ad
	$$
	which implies $x_1-x_4=y_1-y_4$. The claim follows. 
\end{proof}

\begin{lem}\label{lemma:rels 2  chain}
	Let $A=kQ/I$ be a non-wild finite-dimensional algebra with $\rad^3(A)=\{0\}$, and assume 
	$C=((a_1,b_1),(a_2,b_2))$ is a surjective Kronecker chain in $A$ with $t(a_2)\neq s(a_1)$ (that is, $C$ is non-circular). Then $A$ is isomorphic by base change to $kQ/J$ for some admissible ideal $J$ such that 
	$$
		a_1a_2, b_1b_2, a_1b_2+b_2a_1 \in J 
	$$  
	Moreover, we have $\Delta_{kQ/J}(a_1,b_1)= \Delta_{kQ/J}(a_2,b_2)$, which implies that $\Delta_A(a_1,b_1)$ and $\Delta_A(a_2,b_2)$ are conjugate.
\end{lem}
\begin{proof}
	Let $e=e_1+e_2+e_3$, where $e_1$, $e_2$ and $e_3$ are the standard idempotents in $A$ corresponding to the vertices $s(a_1)$, $t(a_1)$ and $t(a_2)$, respectively. As $A$ is non-wild,
	it follows that $\dim_k (e_1 A e_2) \leq 2$ and $\dim_k (e_2 A e_3) \leq 2$. Combined with Lemma \ref{rem:shape}, this shows the algebra $B = eAe/Ae_3Ae_1A $ has quiver 
	\begin{equation}
	Q':
	\begin{tikzcd}
	1 \ar[shift left]{r}{a_1} \ar[shift right,swap]{r}{b_1} & 2 \ar[shift left]{r}{a_2} \ar[shift right,swap]{r}{b_2} & 3 
	\end{tikzcd}
	\end{equation}
	Note that any derivation of $A$ sending the standard idempotents to zero will stabilise $e_3Ae_1$ and, as a consequence, $Ae_3Ae_1A$.
	So in fact the natural epimorphism $eAe\twoheadrightarrow B$ induces a map $\HHH^1_{\rad}(eAe)\longrightarrow \HHH^1_{\rad}(B)$, and thus
	$\Delta_{eAe}(a_i,b_i)$ is obtained from $\Delta_B(a_i,b_i)$ by precomposing with this map. The map $\Delta_{A}(a_i,b_i)$ factors through $\Delta_{eAe}(a_i,b_i)$ by Proposition~\ref{prop: Delta}.
	 Hence, if either $\Delta_A(a_1,b_1)$ or $\Delta_A(a_2,b_2)$ is surjective (which is the case by assumption), then either $\Delta_B(a_1,b_1)$  or $\Delta_B(a_2,b_2)$ is surjective.
	 
	  As $B$ is non-wild, Lemma~\ref{lemma double kronecker} implies that $B$ is isomorphic by base change to $kQ'/J'$, where $J'\supseteq (a_1a_2, b_1b_2, a_1b_2+b_1a_2)$, and if this inclusion is an equality, then $\Delta_{kQ'/J'}(a_1,b_1)=\Delta_{kQ'/J'}(a_2,b_2)$ by Proposition \ref{prop: deltas are equal}. It follows that $A$ is isomorphic by base change to $kQ/J$ for an admissible ideal $J$ such that $a_1a_2, b_1b_2, a_1b_2+b_1a_2\in e_1(kQ/J)e_3(kQ/J)e_1(kQ/J)e_3 \subseteq \rad^3(kQ/J)=\{0\}$. I. e.  $J\supseteq (a_1a_2, b_1b_2, a_1b_2+b_1a_2)$. If $J'$ in addition contains the relation $a_1b_2$, then so does $J$, contradicting the definition of a Kronecker chain (specifically condition \eqref{Kc:C}). Hence $J'=(a_1a_2, b_1b_2, a_1b_2+b_1a_2)$, which implies that $\Delta_{kQ'/J'}(a_1,b_1)=\Delta_{kQ'/J'}(a_2,b_2)$. As $\Delta_{kQ/J}(a_i,b_i)$ factors through $\Delta_{kQ'/J'}(a_i,b_i)$ for $i=1,2$, it follows that 
	 $\Delta_{kQ/J}(a_1,b_1)=\Delta_{kQ/J}(a_2,b_2)$, as required.
\end{proof}

The second part of Theorem~\ref{thm:kronecker main} (in the radical cubed zero case) follows immediately from the above, with the exception of a circular Kronecker chain of length two. 

\begin{cor}\label{cor: all delta surjective}
	Let $A=kQ/I$ be a non-wild finite-dimensional algebra with $\rad^3(A)=0$. Assume that 
	$C=((a_1,b_1),\ldots,(a_n,b_n))$ is a surjective maximal Kronecker chain in $A$. 
	Assume moreover that either $n>2$ or $t(a_n)\neq s(a_1)$.
	Then all $\Delta_A(a_i, b_i)$ are conjugate, and therefore surjective. 
\end{cor}
\begin{proof}
	By assumption at least one of the $\Delta_A(a_i, b_i)$ is surjective. Now we can 
	apply Lemma~\ref{lemma:rels 2  chain} to the Kronecker chains of length two containing $(a_i,b_i)$, and repeat this process as often as required, establishing in each step that two neighbouring $\Delta_A(a_j,b_j)$'s are conjugate.
\end{proof}

\subsection{Establishing \ref{cond:S1}.}
\label{sec:s1}
Our next goal is to establish the relations \ref{cond:S1}.  By the above, it will be sufficient to do so in the case where the radical cubes to zero.

\begin{lem}\label{lem:relations s1}
	Let $A=kQ/I$ be a non-wild finite-dimensional algebra with $\rad^3(A)=0$. Assume that 
	 $C=((a_1,b_1),\ldots,(a_n,b_n))$ is a surjective maximal Kronecker chain in $A$. Then the composition of any arrow in $C$ with any arrow not in $C$ is zero.
\end{lem}
\begin{proof}
	As we have seen in Lemma \ref{rem:shape}, no arrow outside $C$ may share its source or target with any of the $a_i$'s and $b_i$'s.
	So the only arrows in $C$ which could have non-trivial product with an arrow not in $C$  are $a_1, b_1$ (from the left) and $a_n, b_n$ (from the right). In particular we may assume that $s(a_1)\neq t(a_n)$, as the claim is trivial otherwise.
	
	Assume $c$ is an arrow outside $C$.
	 Let us first assume $s(c)=t(a_n)=t(b_n)$. Let $e_n$, $e_{n+1}$ and $e_{n+2}$ be the idempotents in $A$ corresponding to $s(a_n)$, $t(a_n)$ and $t(c)$, respectively. Set $e=e_n+e_{n+1}+e_{n+2}$ if $e_n\neq e_{n+2}$, and $e=e_n+e_{n+1}$ otherwise.
	Clearly $a_n$ and $b_n$ are the only arrows emanating from $e_n$, and the only arrows terminating at $e_{n+1}$. If $c$ is not the only arrow from $s(c)$ to $t(c)$ then our claim must hold, since otherwise our Kronecker chain would not be maximal. So let us assume that $c$ is the only arrow from $s(c)$ to $t(c)$ .
	
	Set $B=A/(A(1-e)A, Ae_{n+2}Ae_{n}A)$ if $e_{n+2} \neq e_n$ and $B=A/A(1-e)A$ if $e_{n+2}=e_n$.
	
	Since $A$ is not wild, the algebra $B$ is not wild either. Note that by the above discussion $a_n$, $b_n$ and $c$ are the only arrows of $Q$ that do not map to zero in $B$. So $B$ has quiver $Q'$ looking as follows
	$$
		\begin{tikzcd}
			n \ar[shift left]{r}{a_n} \ar[shift right,swap]{r}{b_n} & n+1 \ar{r}{c}  & n+2 
		\end{tikzcd}\quad
		\textrm{ or }\quad
		\begin{tikzcd}
		\phantom{+1}n \ar[shift left]{r}{a_n} \ar[shift right,swap]{r}{b_n} & n+1 \ar[bend right, swap]{l}{c}  
		\end{tikzcd}
	$$ 
	In the first case $kQ'$ is wild, so $B$ needs to be a quotient of 
	$kQ'/((x_1a_n + x_2b_n)c)$ for appropriately chosen $(0,0)\neq (x_1,x_2)\in k^2$.
	As $e_nA(1-e)Ae_n \subseteq \rad^3(A)$ (any such path in $Q$ starting in $e_n$ needs to pass through $e_{n+1}$ and some $e_i$, for $i \neq n,n+1,n+2$) and $e_nAe_{n+2}Ae_n \subseteq \rad^3(A)$ (for the same reason), we get from our relation in $B$ the relation $(x_1a_n+x_2b_n)c=0$ in $A$ (as $\rad^3(A)=0$). Similarly, in the second case,  by \cite[Theorem 1]{brustle2001tame}, $B$ again must be a quotient of 
	$kQ'/((x_1a_n+x_2b_n)c)$ for $(0,0)\neq (x_1,x_2)\in k^2$. Again $e_nA(1-e)Ae_n \subseteq \rad^3(A)$, since any path starting in $e_n$ must pass through $e_{n+1}$ (and then, to be in the specified set, pass through a different vertex before returning to $e_n$). Hence we get the relation $(x_1a_n+x_2b_n)c=0$ in $A$ regardless of the case we are in.
	
	By the surjectivity of $\Delta_A(a_n,b_n)$ (which follows by Corollary~\ref{cor: all delta surjective}) we get, for any $y_1,y_2\in k$, a derivation $\delta$ on $A$ such that $\delta(a_n)= wa_n + y_1b_n$, 
	$\delta(b_n)=y_2a_n+wb_n$ for some $w\in k$. We will have $\delta(c)=zc$ for some $z\in k$, as no other arrow shares both source and target with $c$. Therefore
	$$
		\begin{array}{rcl}
		0&=&\delta((x_1a_n+x_2b_n)c)= (w+z)(x_1a_n+x_2b_n)c + x_1y_1 b_n c + x_2y_2a_nc \\&=&
		(x_1y_1 b_n + x_2y_2a_n)c
		\end{array}
	$$
	As we may choose $y_1,y_2$ freely we get the relations $a_nc=0$ and $b_nc=0$ in $A$ (even in the special case where $x_1=0$ or $x_2=0$).
	
	Now we would also have to deal with the case where $t(c)=s(a_1)=s(b_1)$. However, this case is completely analogous to the above (technically it is even covered by the above if we replace $A$ by its opposite algebra, using \cite[Ch. XIX, Lemma 1.4]{simson2006elements} again).
\end{proof}
	
\subsection{Establishing \ref{cond:S2}--\ref{cond:S3} globally}	
	\label{sec:s2s3glo}
We will now develop the tools to get the relations \ref{cond:S2}--\ref{cond:S3} globally, using a type of patching argument based on the conjugacy of the representations $\Delta_A(a_i,b_i):\ \HHH^1_{\rad}(A) \longrightarrow \mathfrak{sl}_2$ (for $(a_i,b_i)$ in a Kronecker chain) that was established in Corollary \ref{cor: all delta surjective}. The following proposition is the tool to turn relations ``holding up to isomorphism'' into genuine relations. 

\begin{prop}\label{prop: delta normalises relations}
	Let $A = kQ/I$, where 
	\begin{equation}
	Q:
	\begin{tikzcd}
	1 \ar[shift left]{r}{a} \ar[shift right,swap]{r}{b} & 2 \ar[shift left]{r}{c} \ar[shift right,swap]{r}{d} & 3 
	\end{tikzcd} 
	\textrm{ or }
	\begin{tikzcd}
	1 \ar[shift left]{r}{a} \ar[shift right,swap]{r}{b} & 2\ar[shift left, bend left=50]{l}{d} \ar[shift right,swap, bend right=50]{l}{c}  
	\end{tikzcd}
	\end{equation}
	and $I$ is an admissible ideal. Assume that $A$ is isomorphic by base change to $kQ/J$, where $J\supseteq (ac,bd,ad+bc)$ and $J\nsupseteq (ac,bd,ad,bc)$, and assume that $\Delta_A(a,b)$ is surjective and  $\Delta_A(a,b)=\Delta_A(c,d)$. 
	Then $I\supseteq (ac,bd,ad+bc)$.
\end{prop}
\begin{proof}
	Fix $J\supseteq (ac,bd,ad+bc)$, and fix an isomorphism by base change $\varphi:\ A\longrightarrow kQ/J$. We necessarily have $\varphi=T_X(a,b)\circ T_Y(c,d)$ for certain $X,Y\in \operatorname{GL}_2(k)$, as $a,b,c,d$ are the only arrows in $Q$.
	Note that 
	$\Delta_{kQ/J}(a,b)=\Delta_{kQ/J}(c,d)$ by Proposition~\ref{prop: deltas are equal}. 
	
	By Proposition~\ref{prop:delta base change} we have
	$$\Delta_{A}(a,b)(\varphi^{-1}\circ\delta\circ \varphi) = X \cdot \Delta_{kQ/J}(a,b)(\delta)\cdot X^{-1}$$ 
	and 
	$$\Delta_{A}(c,d)(\varphi^{-1}\circ\delta\circ \varphi) = Y \cdot \Delta_{kQ/J}(c,d)(\delta)\cdot Y^{-1}$$ 
	for any derivation $\delta$ of $kQ/J$. Due to the equalities of $\Delta$'s we thus obtain 
	$$(X^{-1}Y)\cdot \Delta_{A}(a,b)(\delta)\cdot (Y^{-1}X)= \Delta_{A}(a,b)(\delta)$$ 
	for all derivations 
	$\delta$ of $A$. As $\Delta_{A}(a,b)$ is surjective onto $\mathfrak{sl_2}$ (by assumption) this implies $X=z\cdot Y$ for some 
	$z\in k$. Write $x_1,\ldots, x_4$ for the entries of $X$. Then 
	$$
	\varphi(ac) = (x_1a+x_2b)(zx_1c+zx_2d)= zx_1^2ac+zx_1x_2(bc+ad)+ zx_2^2bd = 0
	$$
	$$
	\varphi(bd) = (x_3a+x_4b)(zx_3c+zx_4d)= zx_3^2ac+zx_3x_4(bc+ad)+ zx_4^2bd = 0
	$$
	and 
	$$
	\begin{array}{rcl}
	\varphi(ad+bc) &=& (x_1a+x_2b)(zx_3c+zx_4d) + (x_3a+x_4b)(zx_1c+zx_2d)\\
	&=& zx_1x_4ad+zx_2x_3bc + zx_2x_3ad + zx_1x_4 bc =0 
	\end{array}
	$$
	Hence $(ac,bd,ac+bd) \subseteq \ker(\varphi)$, which implies $I\supseteq (ac,bd,ac+bd)$, as $\varphi$ is an isomorphism by definition.
\end{proof}

\begin{cor}\label{cor:rels in chain}
	Let $A = kQ/I$ be a non-wild finite-dimensional algebra with ${\rad^3(A)=\{0\}}$, where 
	\begin{equation}
	Q:
	\begin{tikzcd}
	1 \ar[shift left]{r}{a_1} \ar[shift right,swap]{r}{b_1} & 2 \ar[shift left]{r}{a_2} \ar[shift right,swap]{r}{b_2} & \cdots   n\ar[shift left]{r}{a_{n}} \ar[shift right,swap]{r}{b_{n}} & n+
	1
	\end{tikzcd}
	\end{equation}
	for $n \geq 1$
	and $I$ is an admissible ideal.
	Let us allow for vertex $1$ to be equal to vertex $n+1$ if $n\geq 3$, but all other vertices must be distinct no matter what.

	If $C=((a_1,b_1),\ldots, (a_n,b_n))$ is a surjective Kronecker chain in $A$, then $A$ is isomorphic by base change to an algebra $kQ/J$ in which
	$a_ia_{i+1}=0$, $b_ib_{i+1}=0$ and $a_ib_{i+1}+b_ia_{i+1}=0$ for $i=1,\ldots,n$, and $a_na_1=0$, 
	$b_nb_1=0$ and $a_nb_1+b_na_1=0$.
\end{cor}
\begin{proof}
	By Corollary~\ref{cor: all delta surjective} we know that all $\Delta_A(a_i,b_i)$ are conjugate. 
	Using Proposition~\ref{prop:delta base change} we can then find 
	$Y_1,\ldots, Y_n \in \operatorname{GL}_2(k)$ such that $\Delta_{kQ/J}(a_i,b_i)=\Delta_{kQ/J}(a_j,b_j)$ for all $1\leq i,j\leq n$, where 
	$J$ is the preimage of $I$ under $\varphi = T_{Y_1}(a_1,b_1)\circ \cdots \circ T_{Y_n}(a_n,b_n)$.
	Note that $\varphi$ induces an isomorphism by base change  between $A$ and $B=kQ/J$. 
	
	Fix $1 \leq i \leq n-1$, set $e=e_i+e_{i+1}+e_{i+2}$, and consider 
	$$
		B_i = eBe/ eB e_{i+2} B e_i B e
	$$
	As $B$ is non-wild, so is $B_i$, and therefore $a_i$, $b_i$, $a_{i+1}$ and $b_{i+1}$ are the only arrows in the quiver of $B_i$. From Lemma~\ref{lemma:rels 2  chain} we know that 
	$a_ia_{i+1}$, $b_ib_{i+1}$ and $a_ib_{i+1}+b_ia_{i+1}$ are relations in $B$ up to isomorphism by base change, so certainly the same is true in $B_i$. Lastly, $\Delta_B(a_i,b_i)$ factors through $\Delta_{B_i}(a_i, b_i)$, since any derivation of $B$ sending the standard idempotents to zero stabilises $eBe$ and $eB e_{i+2} B e_i B e$. In the same vein $\Delta_B(a_{i+1},b_{i+1})$ factors through
	$\Delta_{B_i}(a_{i+1}, b_{i+1})$. As $\Delta_B(a_i,b_i)=\Delta_B(a_{i+1},b_{i+1})$ it follows that  $\Delta_{B_i}(a_i,b_i)$ and $\Delta_{B_i}(a_{i+1}, b_{i+1})$ coincide on derivations of $B_i$  which come from derivations of $B$. By surjectivity of $\Delta_B(a_i,b_i)$ it follows that if
	$\Delta_{B_i}(a_{i},b_{i})$ and $\Delta_{B_i}(a_{i+1},b_{i+1})$ are conjugate, then they must be equal. 
	By Lemma~\ref{lemma double kronecker} $\Delta_{B_i}(a_{i},b_{i})$ and $\Delta_{B_i}(a_{i+1},b_{i+1})$ could only be non-conjugate if $\{0\}=\rad^2(B_i)=e_iB_i e_{i+2}$, which is equivalent to $e_iBe_{i+2} \subseteq e_i B e_{i+2} B e_i B e_{i+2} \subseteq \rad^3(B)=\{0\}$. But $e_i B e_{i+2}=\{0\}$ (which implies $e_i A e_{i+2}=\{0\}$) is not allowed if $C$ is to be a Kronecker chain. Hence $\Delta_{B_i}(a_i,b_i)$ and $\Delta_{B_i}(a_{i+1}, b_{i+1})$ are conjugate and therefore equal.
	
	We can now apply Proposition~\ref{prop: delta normalises relations} to obtain that $a_ia_{i+1}$, 
	$b_ib_{i+1}$ and $a_ib_{i+1}+b_ia_{i+1}$ are zero in $B_i$, which implies that, as elements of $B$, they are contained in $e_i B e_{i+2} B e_i B e_{i+2}\subseteq \rad^3(B)=\{0\}$. Hence 
	$a_ia_{i+1}=0$, 
	$b_ib_{i+1}=0$ and $a_ib_{i+1}+b_ia_{i+1}=0$ in $B$, and this is true for all $i$.  
	
	The only thing left to show is that $a_na_1=0$, 
	$b_nb_1=0$ and $a_nb_1+b_na_1=0$ in $B$. If $((a_2,b_2), \ldots, (a_n,b_n),(a_1,b_1))$ is not a Kronecker chain, then these relations hold trivially. Otherwise we may apply the same argument as before to this Kronecker chain, which gives us the relations (noting that there is no need to repeat the step in which we apply an isomorphism by base change).  
\end{proof}

The above excludes the case of a circular Kronecker chain of length two, which is a special case mainly for technical reasons. The following proposition deals with that special case.

\begin{prop}\label{prop: special case rels}
	Assume $A=kQ/I$ is a finite-dimensional non-wild $k$-algebra with $\rad^3(A)=\{0\}$, where 
\begin{equation}
Q:
\begin{tikzcd}
1 \ar[bend left,shift left=3]{r}{a} \ar[bend left,shift left,swap]{r}{b} & 
2\ar[bend left,shift left=3]{l}{d} \ar[bend left,shift left,swap]{l}{c} 
\end{tikzcd}
\end{equation}
	and assume that $\Delta_A(a,b)$ is surjective.
	Then $A$ is isomorphic by base change to $kQ/J$, where $J\supseteq (ac, bd, ad+bc)$. If $J\nsupseteq (ac,bc,ad,bd)$, then  $\Delta_{kQ/J}(a,b)=\Delta_{kQ/J}(c,d)$, and 
	$J \supseteq (ca, db, cb+da)$.
\end{prop}
\begin{proof}
	By \cite[Theorem 1]{brustle2001tame} and the fact that $\rad^3(A)=\{0\}$, the algebra $A$ is isomorphic to $kQ/J$ where
	\begin{enumerate}
		\item $ac \in J$
		\item $bd \in J$ 
		\item $(x_1 c + x_2 d)(y_1 a + y_2 b) \in J$
		\item $(x_3 c + x_4 d)(y_3 a +y_4b ) \in J$
	\end{enumerate}
	for $x_i, y_i \in k$ such that $x_1x_4 \neq x_2x_3$ and $y_1y_4 \neq y_2y_3$. Any isomorphism of $A$ is by base change, as any non-zero path in $A$ connecting the two vertices is an arrow.
	Hence we may assume without loss of generality that $I\supseteq (ac,bd)$. 
	
	As we assume $\Delta_A(a,b)$ to be surjective, there is a derivation $\delta$ of $A$ such that $\delta(a)=wa+b$ and $\delta(b)=a+wb$ for some $w\in k$. Write $\delta(c)=z_1c+z_2d$ and $\delta(d)=z_3c+z_4d$. Then 
	$$
	0=\delta(ac) = bc + z_2ad
	$$
	If $z_2$ is non-zero we can apply an isomorphism by base change, and assume without loss of generality that 
	$bc+ad\in I$ (which would conclude the proof of the first part). If $z_2=0$, then $I\supseteq (ac,bd,bc)$, and considering $0=\delta(bd)=ad$, we obtain $I\supseteq (ac,bd,bc,ad)$, so we are also done in this case. 
	
	Now assume $I\supseteq (ac,bd,ad+bc)$, but $I\nsupseteq (ac, bd, ad,bc)$. By Proposition~\ref{prop: deltas are equal} we have $\Delta_A(a,b)=\Delta_A(c,d)$. 
	Applying what we have proved so far with the roles of $a, b$ and $c,d$ reversed shows that there are $X, Y\in \operatorname{GL}_2(k)$ such that $J=T_X^{-1}(a,b)(T_Y^{-1}(c,d)(I))$ contains $ca$, $db$ and $da+cb$. Now Proposition~\ref{prop: delta normalises relations} implies that  $ca$, $db$ and $da+cb$ already lie in $I$, which completes the proof.
\end{proof}

All that is left to do now is see how all of these propositions and lemmas fit together.

\subsection{Proof of Theorem~\ref{thm:kronecker main}}
\label{sec:proofkromain}
Theorem \ref{thm:kronecker main} holds for $A$ if and only if it holds for every connected component of $A$, hence we may assume $A$ is connected.	
	By Corollary~\ref{cor: reduction to rad3 zero} we may assume $\rad^3(A)=\{0\}$
	Let us start with the first part. If $C$ is a surjective maximal Kronecker chain in $A$, then by 
	Lemma~\ref{lem:relations s1} the relations \ref{cond:S1} are satisfied. This will remain to be the case in any algebra isomorphic to $A$ by base change. Now, given that the relations \ref{cond:S1} hold, the relations \ref{cond:S2}--\ref{cond:S3} hold in $A$ if and only if they hold in $eAe/eJe$, where $e$ is the sum of all standard idempotents associated with either the source or the target of one of the $a_i$, and $J$ is the ideal generated by all arrows not in $C$. The algebra
	$eAe/eJe$ is of the right form to apply either Corollary~\ref{cor:rels in chain} or Proposition~\ref{prop: special case rels} to it. Either way, it follows that there is an algebra isomorphic by base change to $eAe/eJe$ in which the relations \ref{cond:S2}--\ref{cond:S3} hold. By extending this base change by the identity we get that $A$ is isomorphic by base change to an algebra in which  \ref{cond:S2}--\ref{cond:S3} hold, in addition to \ref{cond:S1} (which is not affected by isomorphism by base change).
	
	In the other direction, if \ref{cond:S1}--\ref{cond:S3} are satisfied in an algebra isomorphic by base change to $A$, then $C$ is surjective by Proposition~\ref{prop:trivial if} (as we noted before).
	The second part of the theorem follows from either Corollary~\ref{cor: all delta surjective} or Proposition~\ref{prop: special case rels}, depending on $C$.

\subsection{Proof of Theorem \ref{thm:mainmain}}
\label{sec:proofmainmain}
There is an exact sequence of Lie algebras
\begin{equation}
0 \to \mathfrak{r}'' \to \HHH^1_{\rad}(A) \xrightarrow{\varphi \circ \pi} \HHH^1_{\rad}(kQ^s),
\end{equation}
where $\mathfrak{r}''=\ker(\varphi \circ \pi)$ is solvable by Corollary~\ref{cor:iteration} and Corollary~\ref{cor:solv-rad}. Since $A$ is non-wild, (the proof of) Proposition \ref{prop:solv-dynkin} shows that $\HHH^1(kQ^s) \cong \mathfrak{sl}_2^{\oplus n} \oplus \mathfrak{r'}$, where $n$ is the number of connected components of $Q^s$ isomorphic to $K_2$, and $\mathfrak{r'}$ is solvable. By definition, the composition of $\varphi \circ \pi$ with one of the projections to $\mathfrak{sl}_2$ is equal to one of the maps $\Delta_A(a,b)$ for certain arrows $a,b\in A$ (see Proposition~\ref{prop: Delta}). 
Hence we get an exact sequence of Lie algebras
\begin{equation}
0 \to \mathfrak{r} \to \HHH^1_{\rad}(A) \xrightarrow{\bigoplus_{i} \Delta_A(a_i,b_i)} \mathfrak{sl}_2^{\oplus n},
\end{equation}
where $\mathfrak r$ is solvable and $(a_1,b_1), \ldots, (a_n,b_n)$ are all pairs of arrows in $A$ for which a map $\Delta_A(a_i,b_i)$ is defined (i.e. pairs of arrows whose corresponding arrows in $Q^s$ form a $K_2$).

By Theorem \ref{thm:kronecker main} \eqref{statement-2} the image of the rightmost map is isomorphic to a subalgebra of $\mathfrak{sl}_2^{\oplus m} \oplus \mathfrak{r'''}$, where $m$ is the number of maximal surjective Kronecker chains, which by Theorem \ref{thm:kronecker main} \eqref{statement-1} coincides with the number of equivalence classes of maximal Kronecker chains with standard relations, and $\mathfrak r'''$ is solvable. More precisely, we get a homomorphism of Lie algebras
\begin{equation}\label{eq: delta i0}
\bigoplus_{i\in I_0} \Delta_A(a_i,b_i):\ \HHH^1_{\rad}(A) \longrightarrow \mathfrak{sl}_2^{\oplus m},
\end{equation}
with solvable kernel, where $I_0$ contains, for each equivalence class of maximal Kronecker chains with standard relations, exactly one $i$ such that $(a_i,b_i)$ is involved in that Kronecker chain.
Now Proposition~\ref{prop:trivial if} provides, for each $i\in I_0$, a Lie algebra $\mathfrak{g_i}\subset \HHH^1_{\rad}(A)$ such that $\Delta_A(a_i,b_i)$ maps $\mathfrak{g}_i$ isomorphically onto $\mathfrak{sl}_2$ (and $\Delta_A(a_j,b_j)(\mathfrak g_i)=\{0\}$ for any $i\neq j \in I_0$, which follows from the definition of the $\mathfrak g_i$). It follows that $\HHH^1_{\rad}(A)$ is the direct sum of all of the $\mathfrak{g_i}$ (which the map in \eqref{eq: delta i0} maps isomorphically onto $\mathfrak{sl}_2^{\oplus m}$) and a solvable Lie algebra, as claimed.

\subsection{Assorted examples}
	Let us now apply Theorem~\ref{thm:mainmain} to a few typical examples. All of the algebras listed below are special biserial, and therefore tame. They also have trivial centre. The primary purpose of these examples is to illustrate our main results. In addition to that, as a toy application, we use the number of copies of $\mathfrak{sl}_2$ in the first Hochschild cohomology as a derived invariant (notwithstanding the fact that the contents of \S\ref{css} are our main intended application).
	\begin{enumerate}
		\item Let $A=kQ/I$, where 
		$$
		Q:
		\begin{tikzcd}
		1 \ar[shift left]{r}{a} \ar[shift right,swap]{r}{b} & 
		2  \ar[shift left]{r}{c} \ar[shift right,swap]{r}{d}& 3
		 \ar{r}{e} & 4
		\end{tikzcd}
		$$
		and $I=(ac, bd, ad+bc, ce, de)$.
		 Then $((a,b), (c,d))$ is a maximal Kronecker chain with standard relations, and by Proposition~\ref{prop:der-rad} and Theorem~\ref{thm:mainmain}
		 $$
		 	\HHH^1(A)=\HHH^1_{\rad}(A)\cong \mathfrak{sl}_2 \oplus \mathfrak{r}
		 $$
		 for a solvable Lie algebra $\mathfrak{r}$. One can compute that in this example $\mathfrak r = 0$, and therefore $\dim (\HHH^1(A))=3$.
		 \item  Let $B=kQ/I$, where 
		 $$
		 Q:
		 \begin{tikzcd}
		 1 \ar[shift left]{r}{a} \ar[shift right,swap]{r}{b} & 
		 2  \ar[shift left]{r}{c} \ar[shift right,swap]{r}{d}& 3
		  \ar[shift left]{r}{e} \ar[shift right,swap]{r}{f} & 4
		 \end{tikzcd}
		 $$
		 and $I=(ac, bd, ce, df)$. The maximal Kronecker chain $((a,b),(c,d),(e,f))$ does not have standard relations, as (among other things) no linear combination of $ad$ and $bc$ lies in $I$. Therefore $\HHH^1_{\rad}(B)$ is solvable, and  $\HHH^1(B)=\HHH^1_{\rad}(B)$ by Proposition~\ref{prop:der-rad}. One can again compute $\dim(\HHH^1(B))=3$, but as $\HHH^1(B)$ is solvable, this algebra cannot be derived equivalent to the algebra $A$ from the previous point. 
		 \item  Let $C=kQ/I$, where 
		 $$
		Q:
\begin{tikzcd}
1 \ar{r}{a}\ar[out=135, in=-135, loop, swap]{}{c} & 
2  \ar{r}{b} \ar[out=45, in=135, loop, swap]{}{d} & 3 \ar[out=45, in=-45, loop]{}{e}
\end{tikzcd}
$$		
			and $I=(ab, c^2, d^2, e^2, cad, adb, dbe)$. It follows already from 
			Theorem~\ref{thm:main} that $\HHH^1_{\rad}(C)$ is solvable. Since we assume $\cha(k)\neq 2$ one has $\HHH^1(C)=\HHH^1_{\rad}(C)$, and one can compute $\dim(\HHH^1(C))=4$.
		 \item Let $D=kQ/I$, where 
		 	$$
		 Q:
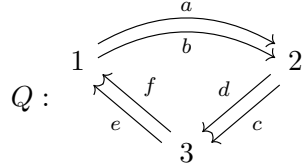
		 \begin{tikzcd}
		 1 \ar[shift left, bend left]{rr}{a} \ar[shift right, bend left, swap]{rr}{b} & &
		 2  \ar[shift left]{ld}{c} \ar[shift right,swap]{ld}{d} \\ & 3
		 \ar[shift left]{lu}{e} \ar[shift right,swap]{lu}{f} 
		 \end{tikzcd}
		 $$
		 and $I=(ac, bd, ad+bc, ce, df, cf+de, ea, fb, eb+fa)$. Then $((a,b), (c,d), (e,f))$ is a maximal Kronecker
		 chain with standard relations, and therefore 
		 $\HHH^1(D)=\HHH^1_{\rad}(D)=\mathfrak{sl}_2\oplus \mathfrak{r}$ for a solvable Lie algebra $\mathfrak{r}$. One can compute that  $\dim (\HHH^1(D))=4$, and therefore $\mathfrak{r}=\mathfrak{gl}_1$. But as $\HHH^1(D)$ is non-solvable, the algebra $D$ cannot be derived equivalent to the algebra $C$ from the previous point (although in this case this also already follows from the Cartan determinant).
		 
		 \item Let the quiver $Q$ be as in the previous point, and let $J$ be the ideal in $kQ$ generated by all paths of length two. Set $E=kQ/J$. Then $((a,b))$, $((c,d))$ and $((e,f))$ are  maximal Kronecker chains of length one with standard relations, and therefore
		 $\HHH^1(E)=\HHH^1_{\rad}(E)=\mathfrak{sl_2}^{\oplus 3}\oplus \mathfrak{r}$ for a solvable Lie algebra $\mathfrak{r}$, and once again one can compute $\mathfrak{r}=\mathfrak{gl}_1$.
	\end{enumerate}

\section{On a question by Chaparro, Schroll and Solotar}
\label{css}
The following question was proposed by Chaparro, Schroll and Solotar~\cite{chaparro2018lie}.

\begin{ques}
\label{question-css}
Is it true that, except in some low dimensional cases, the first Hochschild cohomology space of any finite-dimensional algebra of tame representation type is a solvable Lie algebra?
\end{ques}

Interpreted literally, the following example shows that this question has a negative answer, by exhibiting tame algebras of arbitrarily high dimension which have non-solvable $\HHH^1$.

\begin{ex}
For any $2 \leq n\in \NN$, consider the finite-dimensional algebra $A_n=kQ_n/\rad(kQ_n)^2$, where
\begin{equation}
Q_n:
\begin{tikzcd}
1 \ar[shift left]{r} \ar[shift right]{r} & 2 \ar{r} & \cdots  \ar{r} & n-1 \ar{r}  & n
\end{tikzcd}
\end{equation}
Then $A_n$ is a radical square zero algebra of tame representation type by Proposition~\ref{prop:rad-square}~\eqref{eq:tame}, since the separated quiver $Q^s$ of $A_n$ is the disjoint union of a Kronecker quiver and $n-2$ copies of $A_2$. Moreover,
\begin{equation}
\HHH^1(kQ^s) \cong \mathfrak{sl}_2
\end{equation}
is not solvable, so by Corollary \ref{cor:solv-rad}, neither is $\HHH^1_{\rad}(A_n)=\HHH^1(A_n)$.
\end{ex}

Other examples can be constructed using Theorem \ref{thm:mainmain}. More importantly perhaps, there exist large classes of algebras for which Question \ref{question-css} does have a positive answer. Indeed, one of the main results in \cite{chaparro2018lie} is the following.

\begin{thm}\cite[Thm. 5.4]{chaparro2018lie}
\label{thm:gentle}
Let $A$ be a gentle $k$-algebra, and $\cha(k)=p\neq 2$. Then $\HHH^1(A)$ is solvable if and only if $A$ is not the Kronecker algebra.
\end{thm}

Let us sketch how to recover this result from Theorem \ref{thm:mainmain}, which applies since gentle algebras are tame. First, note that for a gentle algebra $A=kQ/I$ (where we can assume $Q$ connected), the ideal $I$ is generated by paths of length two  by definition. Hence for every loop $a$ in $Q$ we have $a^2=0$. Since $p\neq 2$, Proposition \ref{prop:der-rad} applies and $\HHH^1_{\rad}(A)=\HHH^1(A)$. To prove Theorem \ref{thm:gentle}, it hence suffices to show that for $A$ not the Kronecker algebra, there is no maximal Kronecker chain with standard relations. 

Assume there does exist such a chain.
 By condition \ref{cond:S1} (which holds even without applying an isomorphism by base change), we can assume that the algebra is equal to the maximal surjective Kronecker chain, for otherwise there would exist an arrow $c$ in $Q$ such that there are two different arrows $a, b$ with either $ca=cb=0$ (if $t(c)=s(a)=s(b)$) or  $a c=bc=0$ (if $s(c)=t(a)=t(b)$), contradicting the definition of a gentle algebra.
 Since $A$ is not the Kronecker quiver, we can further assume that the length of the chain is $n>1$.
 Let $((a,b),(c,d))$ be a sub-chain. Conditions \ref{cond:S2} and \ref{cond:S3} show that
 the $k$-vector space spanned by $ac$, $ad$, $bc$ and $bd$ is at most one-dimensional (this property is left invariant by isomorphism by base change). However, in a gentle algebra precisely two of these monomials must be relations, that is,  the dimension of this space is exactly two, a contradiction.

Another class of algebras for which Question \ref{question-css} has a positive answer are the toupie algebras. Toupie algebras form a class of special multiserial algebras which combine features of canonical and monomial algebras. More precisely,  the quiver $Q$ of a toupie algebra  has a unique source and a unique sink and any other vertex is the source of exactly one arrow and the target of exactly one arrow.  

\begin{thm} \cite[Corollary 6.9]{artenstein2018gerstenhaber}.
\label{prop:toupie}
Let $A$ be a toupie algebra different from the Kronecker quiver. If $A$ is of representation finite type or tame, then $\HHH^1(A)$ is solvable.
\end{thm}

Since a toupie algebra $A$ has no loops, by Proposition \ref{prop:der-rad}, $\HHH^1_{\rad}(A)=\HHH^1(A)$. The non-wild toupie algebras on two vertices are $kA_2$ and $kK_2$, and for toupie algebras with more than two vertices, it follows from the description of $Q$ that $\overline{Q^s}$ does not contain an $\tilde{A_1}$, hence Theorem \ref{thm:main} applies and one recovers Theorem \ref{prop:toupie}.

\subsection{Tame symmetric algebras}

As a new application of Theorem \ref{thm:mainmain}, we now show how tame symmetric algebras also provide a positive answer to Question \ref{question-css}. We can assume $\cha(k)=p \neq 2$, since in this case we already know that $\HHH^1_{\rad}(A)$ is solvable by Theorem \ref{thm:main}\eqref{eq:main-1}.

\begin{thm}
\label{thm:tame-symmetric}
Let $A$ be a symmetric algebra of non-wild representation type, and $p \neq 2$. Then $\HHH^1_{\rad}(A)$ is solvable if and only if $A\not\cong T(K_2)$, the trivial extension of the Kronecker quiver.
\end{thm}
\begin{proof}
If $A \cong T(K_2)$, which is symmetric special biserial, hence tame, then $\HHH^1_{\rad}(A)$ is not solvable, since it has a maximal Kronecker chain of length two with standard relations.

For the other direction, suppose that $C=((a_1,b_1),\ldots,(a_n,b_n))$ is a maximal Kronecker chain with standard relations in $A$, and $A\not\cong T(K_2)$ . 
We first claim that case \eqref{eq:shape3} in Lemma \ref{rem:shape} cannot occur. Suppose it would, then if we denote by $P$ the indecomposable projective module corresponding to the source vertex of $a_1$, condition \ref{cond:S1} ensures that
$\text{top}(P) \ncong \text{soc}(P)$, contradicting the assumption that $A$ is symmetric. 
Next, we claim that if we are in case \eqref{eq:shape2} of Lemma \ref{rem:shape}, then $n=2$. Indeed, assume $n>2$, then by \ref{cond:S2} and \ref{cond:S3}, there is no non-zero cycle in $A$ starting at $s(a_1)$, so again, we find $\text{top}(P) \not\cong \text{soc}(P)$.
If we are in case \eqref{eq:shape1} of Lemma \ref{rem:shape}, then $A$ is a quotient of the exterior algebra on two generators $\Lambda(x,y)=k\langle x,y \rangle/(x^2, y^2, xy+yx)$, which is not symmetric since $p \neq 2$. The only quotients of $\Lambda(x,y)$ which are symmetric are $k$ and $k[x]/(x^2)$, which in particular have a solvable $\HHH^1_{\rad}$.
If we are in case \eqref{eq:shape2} for $n=2$ of Lemma \ref{rem:shape}, then $A$ is a quotient of $B=kQ/I$, where 
\begin{equation}
Q:
\begin{tikzcd}
1 \ar[bend left,shift left=3]{r}{a_1} \ar[bend left,shift left,swap]{r}{b_1} & 
2\ar[bend left,shift left=3]{l}{a_2} \ar[bend left,shift left,swap]{l}{b_2} 
\end{tikzcd}
\end{equation}
and $I=(a_1a_2,b_1b_2,a_2a_1,b_2b_1,a_1b_2+b_1a_2,a_2b_1+b_2a_1)$. In fact $B\cong T(K_2)$.  Since $T(K_2)$ has no further symmetric quotients with maximal surjective Kronecker chains with standard relations, we are done.
\end{proof}

\subsection{Auslander algebras}

Finally, we mention some examples of a different flavour. Remember that an Auslander algebra $\Gamma$ is a finite-dimensional algebra of $\gldim \Gamma \leq 2$ and $\text{dom. dim } \Gamma \geq 2$. Every Auslander algebra $\Gamma$ arises as the endomorphism algebra $\Gamma=\Gamma_A=\End_A(M)$ of an minimal representation generator $M$ for a representation finite algebra $A$. By \cite[Section 4]{MR2032511} there is an injective morphism
\begin{equation}
\HHH^1(\Gamma_A) \xrightarrow{c} \HHH^1(A),
\end{equation}
and we claim this is compatible with the Lie algebra structure.  Hence, assuming $\cha(k)=0$, it follows from Corollary \ref{cor:fin-rep-type} that $\HHH^1(\Gamma)$ is solvable.

To see the compatibility, note that there is a fully faithful functor
\begin{equation}
-\otimes_A P:\Perf(A) \to \Perf(\Gamma_A),
\end{equation}
for $P=\Hom_A(M,A)$, which by \cite[\S 3.3]{keller2003derived} induces a morphism of Gerstenhaber algebras
\begin{equation}
\label{eq:gerst}
\HHH^*(\Gamma_A) \to \HHH^*(A).
\end{equation}
One checks that the morphism $c$, which is constructed in \cite[Theorem 3.5]{MR2032511} is exactly the degree one component in \eqref{eq:gerst}, so $c$ is a morphism of Lie algebras.

\subsection{A class of special biserial algebras}

In \cite{meinel2018gerstenhaber}, the Gerstenhaber algebra structure on the Hochschild cohomology ring of a specific class of special biserial algebras was explicitly determined. These algebras are defined as follows: consider the quiver $Q$ with $m$ vertices $0, \ldots, m-1$ and $2m$ arrows $a_i:i \to i+1$ and $\overline{a}_i:i+1 \to i$, for $i=0, \ldots, m-1$. Considering the indices modulo $m$, let 
\begin{equation}
A_n=kQ/(a_ia_{i+1},\overline{a}_{i+1} \overline{a}_i,(a_i\overline{a}_i)^n-(\overline{a}_{i-1}a_{i-1})^n).
\end{equation}
Since this algebra is special biserial, it is tame, and Theorem \ref{thm:main} immediately implies that $\HHH^1(A_n)$ is solvable,  which was verified by explicit computation in \cite[Section 9.1]{meinel2018gerstenhaber}. Of course the results of \cite{meinel2018gerstenhaber} are much stronger, and solvability of $\HHH^1$ is only mentioned as a side-note.

\begin{ack}
We would like to heartily thank Geoffrey Janssens for alerting us to Question \ref{question-css}, which was the initial motivation for us to write this paper.
\end{ack}

\bibliography{hochschild-solvable}

\end{document}